\documentclass{article}

\usepackage{comment}

\usepackage[margin=1in]{geometry}
\usepackage[caption=false]{subfig}
\usepackage{graphicx}
\usepackage{url}
\usepackage{algorithm}
\usepackage{algpseudocode}
\usepackage{amsthm}
\newtheorem{theorem}{Theorem}

\newtheorem{lemma}{Lemma}
\newtheorem{proposition}{Proposition}
\newtheorem{corollary}{Corollary}
\newtheorem{assumption}{Assumption}
\newtheorem{remark}{Remark}

\usepackage{amsmath}
\usepackage{amssymb}

\newcommand{\plzcheck}[1]{#1}
\newcommand{\rev}[1]{#1}

\DeclareMathOperator{\sgn}{sgn}
\DeclareMathOperator{\Unif}{Unif}

\DeclareMathOperator{\R}{\mathbb{R}}
\newcommand{\ve}[1]{\mathbf{#1}}
\newcommand{\norm}[2]{\left\|#1\right\|_{#2}}
\newcommand{\inner}[2]{\left\langle #1, #2 \right\rangle}
\newcommand{\abs}[1]{\left|#1\right|}
\newcommand{\EE}{\mathbb{E}}
\newcommand{\parens}[1]{\left(#1 \right)}

\DeclareMathOperator{\trace}{trace}

\renewcommand{\sgn}{\text{sign}}
\newcommand{\eps}{\varepsilon}
\DeclareMathOperator{\Accept}{Accept}
\DeclareMathOperator{\quantile}{quantile}

\newcommand{\corrQone}[1]{Q_{#1}}
\newcommand{\corrQ}[2]{Q_{#1}(#2)}
\newcommand{\uncorrQ}[2]{\tilde{Q}_{#1}(#2)}
\newcommand{\uncorrQone}[1]{\tilde{Q}_{#1}}
\newcommand{\mean}[1]{M(#1)}
\newcommand{\subcorrQ}[3]{Q_{#1}(#2,#3)}
\newcommand{\subuncorrQ}[3]{\tilde{Q}_{#1}(#2,#3)}
\newcommand{\densconst}{D}
\newcommand{\subgaussconst}{K}
\makeatletter
\renewcommand*{\ALG@name}{Method}
\makeatother

\mathchardef\mhyphen="2D

\title{Quantile-based Iterative Methods for Corrupted Systems of Linear Equations\thanks{The authors were partially supported by NSF CAREER $\#1348721$ and NSF BIGDATA $\#1740325$. We also acknowledge sponsorship by Capital FundManagement.}}

\author{Jamie Haddock\thanks{Department of Mathematics, University of California, Los Angeles, Los Angeles, CA}
\and Deanna Needell\footnotemark[2]
\and Elizaveta Rebrova\footnotemark[2]
\and William Swartworth\footnotemark[2]}

\date{}
\begin{document}


\maketitle
\begin{abstract}
Often in applications ranging from medical imaging and sensor networks to error correction and data science (and beyond), one needs to solve large-scale linear systems in which a fraction of the measurements have been corrupted. We consider solving such large-scale systems of linear equations $\ve{A}\ve{x}=\ve{b}$ that are inconsistent due to corruptions in the measurement vector $\ve{b}$. We develop several variants of iterative methods that converge to the solution of the uncorrupted system of equations, even in the presence of large corruptions.  These methods make use of a quantile of the absolute values of the residual vector in determining the iterate update.  We present both theoretical and empirical results that demonstrate the promise of these iterative approaches.   
\end{abstract}



\section{Introduction}

One of the most ubiquitous problems arising across the sciences is that of solving large-scale systems of linear equations. These problems arise in many areas of data science including machine learning,  as subroutines of several optimization methods \cite{boyd2004convex},  medical imaging \cite{Gordon1970,herman1993algebraic}, sensor networks \cite{savvides2001dynamic}, statistical analysis, and many more.  A practical challenge in all of these settings is that there is almost always corruption present in any such large scale data, either due to data collection, transmission, adversarial components, or modern storage systems that can introduce corruptions into otherwise consistent systems of equations.  For such applications, we seek methods that are robust to such corruption but scalable to big data.

In this work, we develop scalable methods for solving corrupted systems of linear equations.  Here, we consider the problem of solving large scale systems of equations $\ve{A} \ve{x} = \tilde{\ve{b}}$ where a subset of equations have been contaminated with arbitrarily large corruptions in the measurement vector, thereby constructing an inconsistent system of equations defined by measurement matrix $\ve{A}$ and observed measurement vector $\ve{b} = \tilde{\ve{b}} + \ve{b}_C$ ($\tilde{\ve{b}}$ being unobserved but corresponding to the desired system of equations and $\ve{b}_C$ being an arbitrary corruption vector of the same dimension).  Our work is motivated by the setting where the uncorrupted system of equations $\ve{A} \ve{x} = \tilde{\ve{b}}$ is highly overdetermined and the number of measurements is very large.

We focus on variants of the popular iterative methods, \emph{stochastic gradient descent} (SGD) or \emph{randomized Kacmarz} (RK), that have gained popularity recently due to their small memory footprint and good theoretical guarantees \cite{strohmer2009randomized,bottou2010large,needell2016stochastic}.  We propose variants of both RK and SGD based upon use of \emph{quantile statistics}.  We focus on proving theoretical convergence guarantees for these variants, but additionally discuss their implementation, and present numerical experiments evidencing their promise.

The SGD method is a widely-used first-order iterative method for convex optimization \cite{robbins1951stochastic}.  The classical method seeks to minimize a separable objective function $f(\ve{x}) = \sum_{i=1}^{m} f_i(\ve{x})$ by accessing (stochastically) selected components of the objective and using a gradient step for this component.  That is, SGD constructs iterates $\ve{x}_k$ given by 
\begin{equation}\label{eq:SGD}
\ve{x}_{k+1} = \ve{x}_k - \gamma_k \nabla f_i(\ve{x}_k)
\end{equation}
where $\gamma_k$ is the learning rate (or step-size) and $i$ is the selected component for the $k$th iteration.
When the objective function $f(\ve{x})$ represents error in the solution of a system of equations, SGD generally updates in the direction of the sampled row, i.e., $\ve{x}_{k+1} - \ve{x}_k = \alpha_k \ve{a}_i$ for some $\alpha_k$ which depends upon the iterate $\ve{x}_k$.  Our variants apply SGD to the \emph{least absolute deviations (LAD)} error and \emph{least squares (LS)} error, $$f(\ve{x}) = \|\ve{A} \ve{x} - \ve{b}\|_1 = \sum_{i=1}^m \left|\inner{\ve{a}_i}{\ve{x}} - b_i\right| \quad \text{ and } \quad f(\ve{x}) = \frac{1}{2}\|\ve{A} \ve{x} - \ve{b}\|^2 = \frac{1}{2}\sum_{i=1}^m \left(\inner{\ve{a}_i}{\ve{x}} - b_i\right)^2,$$ respectively.  For these objectives, the SGD updates \eqref{eq:SGD} take the form 
$$
\ve{x}_{k+1} = \ve{x}_k - \gamma_k \text{sign}(\inner{\ve{a}_i}{\ve{x}_k} - b_i)\ve{a}_i
\quad \text{ and }\quad  \ve{x}_{k+1} = \ve{x}_k - \gamma_k (\inner{\ve{a}_i}{\ve{x}_k} - b_i)\ve{a}_i,$$ respectively, where $\text{sign}(\cdot)$ denotes the function that returns $1$ if its argument is positive and $-1$ otherwise. The RK updates are a specific instance of the SGD updates for the LS error where $\gamma_k = 1/\|\ve{a}_i\|^2$ \cite{needell2016stochastic}; that is 
\begin{equation}
\ve{x}_{k+1} = \ve{x}_k + \frac{b_i - \inner{\ve{a}_i}{\ve{x}_k}}{\|\ve{a}_i\|^2} \ve{a}_i.  \label{eq:RKupdate}
\end{equation}
In \cite{strohmer2009randomized}, the authors showed that when applied to a consistent system of equations with a unique solution $\ve{x}^*$ and with a specific sampling distribution, RK converges at least linearly in expectation. Indeed, denoting $\ve{e}_k := \ve{x}_k - \ve{x}^*$ as the difference between the $k$-th iterate of the method and the exact solution of the system, the method guarantees 
\begin{equation}
\mathbb{E}\|\ve{e}_k\|^2 \le \left(1 - \frac{\sigma^2_{\min}(\ve{A})}{\|\ve{A}\|_F^2}\right)^k \|\ve{e}_0\|^2, \label{eq:RKconvrate}
\end{equation}
where $\|\cdot\|_F$ denotes the Frobenius norm and $\sigma_{\min}(\ve{A})$ the smallest (nonzero) singular value of $\ve{A}$.
Standard SGD results (e.g., \cite{shamir2013stochastic}) provide similar convergence rates for SGD on these objectives when the stepsizes are chosen according to an appropriately decreasing schedule. See Section \ref{sec:RW} below for more details and a discussion of related work.


Here, we consider variants of the SGD and RK methods that converge to the solution of the uncorrupted system even in the presence of large corruptions in the measurement vector $\ve{b}$.  We prove convergence rates in the same form as \eqref{eq:RKconvrate}. It is worth noting that both our experimental and theoretical results illustrate that the size of the corruptions do not negatively impact the convergence of the proposed methods. Our methods will make use of SGD and RK steps but will use a quantile of the residual entries in order to determine the step-size. 

\subsection{Organization}
The rest of our paper is organized as follows.  In the remainder of the introduction, we present our main contributions in Section~\ref{sec:contributions}, discuss related works in Section~\ref{sec:RW}, and briefly describe our notations and give required definitions in Section \ref{sec:notation}. We then provide the detailed pseudocode of our proposed methods, QuantileRK$(q)$ and QuantileSGD$(q)$, in Section~\ref{sec:pseudocode}.  We state and prove our theoretical results in Section~\ref{sec:theor_results}.  Within this section, we highlight some new results for random matrices as useful tools 
in Subsection~\ref{subsec:foundations} and then include the proofs of our main convergence results in Subsections~\ref{subsec:QuantileRK} and \ref{subsec:QuantileSGD}. In Section~\ref{sec:imp}, we discuss several implementation considerations that affect the efficiency and convergence of our proposed methods.  In Section~\ref{sec:experiment}, we empirically demonstrate the promise of our methods with experiments on synthetic and real data.   Finally, we conclude and offer some future directions in Section~\ref{sec:conclusion}.

\subsection{Contributions}\label{sec:contributions} 
In this section, we provide summaries of foundational results we prove in high-dimensional probability, then state our main convergence results for the proposed methods.  
 Our main convergence results rely on the following assumptions about the linear system $\ve{A}\ve{x} = \ve{b}$. Let $\ve{A} \in \mathbb{R}^{m\times n}$ be a random matrix with $m\geq n$. We direct the reader to~\cite{HDP} for the random matrix theory definitions involved; we also provide summaries in Section~\ref{sec:notation}.
\begin{assumption}
\label{incoherent}
All the rows $\ve{a}_i$ of the matrix $\ve{A}$ have unit norm and are independent. Additionally, for all $i\in[m],$ $\sqrt{n}\ve{a}_i$ is mean zero isotropic and has uniformly bounded subgaussian norm $\norm{\sqrt{n}\ve{a}_i}{\psi_2} \leq \subgaussconst.$
\end{assumption} 

\begin{assumption}
\label{bounded_density}
Each entry $a_{ij}$ of $\ve{A}$ has probability density function $\phi_{ij}$ which satisfies $\phi_{ij}(t) \leq \densconst \sqrt{n}$ for all $t\in \R.$  (The quantity $\densconst$ is a constant which we will use throughout when referring to this assumption.)
\end{assumption}

The prototypical example of a matrix satisfying both assumptions is a normalized Gaussian matrix, i.e., a matrix whose rows are sampled uniformly over $S^{n-1}.$ Assumptions \ref{incoherent} and \ref{bounded_density} extract the properties of Gaussian matrices that are required for our theory. As such, our work applies to more general distributions, whenever there is enough independence between the entries of the matrix and their distributions are regular enough.

By Assumptions \ref{incoherent} and \ref{bounded_density}, the matrices we consider will be full rank almost surely so the uncorrupted system \plzcheck{$\ve{A}\ve{x} = \tilde{\ve{b}}$} will always have a unique solution $\ve{x}^*.$ 

\subsubsection{High-dimensional probability results} 
%
Our main convergence guarantees build upon several useful results related to the non-asymptotic properties of random matrices that appear to be new and that may be of independent interest.

In particular, Proposition \ref{sub_conditioning} shows that for a class of random matrices, \emph{all} large enough submatrices 
uniformly have smallest singular values that are at least on the order of $\sqrt{m/n}$.  For matrices which satify Assumptions \ref{incoherent} and \ref{bounded_density}, this generalizes standard bounds on the smallest singular value, but does not follow directly from these bounds.

Proposition \ref{med_bound} is more specialized, but may also be of independent interest.  For a random matrix $\ve{A}$, we show that the $q$-quantile of $\{\abs{\inner{\ve{a}_i}{\ve{x}}}\}$ is well concentrated \textit{uniformly} in $\ve{x}$.  Perhaps surprisingly, $\ve{A}$ does not need to be very tall for this result to hold; a constant aspect ratio suffices.

\subsubsection{Main results}
We first introduce two new methods for iteratively solving linear systems with corruptions and give the formal statements of our main results. 
%
%
The first method we introduce is \textbf{QuantileRK}, which builds upon the RK method. Recall that the iteration of RK given by \eqref{eq:RKupdate} implies that the method proceeds by sampling rows of the matrix $\ve{A}$ and projecting onto the corresponding hyperplane given by the linear constraint.
When some of the entries in $\ve{b}$ are corrupted by a large amount, RK periodically projects onto the associated corrupted hyperplanes and therefore does not converge.  Our solution is to avoid making projections that result in $\norm{\ve{x}_{k+1}-\ve{x}_k}{}$ being abnormally large.
Specifically, for each iterate $\ve{x}_k$ we consider the set of distances from $\ve{x}_k$ to a set of \plzcheck{$t$ sampled} hyperplane constraints.\footnote{\label{subset}In order to allow more efficient implementations, we empirically show that considering a small subset of hyperplanes is sufficient. One could extend the theory to this setting as well, with a slightly more complicated analysis.} We assign a threshold value to be the $q$-quantile of these distances, where $q$ is a parameter of the method. If the distance from $\ve{x}_k$ to the sampled hyperplane is greater than this threshold then the method avoids projecting during that iteration.  Otherwise it projects in the same manner as RK.

Theorem \ref{median_convergence} states that the QuantileRK method converges for random matrices satisfying Assumptions \ref{incoherent} and \ref{bounded_density} above, as long as the fraction of corrupted entries is a sufficiently small constant (which does not depend on the dimensions of the matrix).  
Here and throughout, $c,C,c_1,C_1,\ldots$ denote absolute constants that may denote different values from one use to the next. Variable subscripts on constants will indicate quantities that a given constant may depend on.

\begin{theorem}
\label{median_convergence}
Let the system be defined by random matrix $\ve{A} \in \mathbb{R}^{m \times n}$ satisfying Assumptions \ref{incoherent} and \ref{bounded_density}, with the constant parameters $\densconst$ and $\subgaussconst$.\footnote{\label{note:constants}In other words we do not track the dependencies on $\densconst$ and $\subgaussconst$.} 
Then with probability $1- c \exp(-c_q m)$, the iterates produced by the QuantileRK$(q)$ Method \ref{QuantileRK} with $q \in (0,1)$, where in each iteration the quantile is computed using the full corrupted residual (instead of subsampling, we use $t=m$), and initialized with arbitrary point $\ve{x}_0 \in \R^n$ satisfy $$\EE\left(\norm{\ve{e}_{k}}{}^2\right) \leq \left(1-\frac{C_q}{n}\right)^k \norm{\ve{e}_{0}}{}^2$$ 
as long as the fraction of corrupted entries $\beta =|\textup{supp}(\ve{b}_C)|/m < \min (c_1 \rev{q^2}, 1 - q)$ and $m \ge C n$. (Recall that $\ve{e}_k$ denotes the error vector $\ve{x}_k - \ve{x}^*.$)
\end{theorem}

The second method we introduce is \textbf{QuantileSGD}, which is a variant of SGD in which the step-size used in each iteration is chosen to avoid abnormally large steps.  Specifically, for each iterate $\ve{x}_k$, we consider the set of distances from $\ve{x}_k$ to the \plzcheck{set of $t$ sampled} hyperplane constraints specified by our linear system.\textsuperscript{\ref{subset}}  We choose the step-size as the $q$-quantile of these distances, where $q$ is a parameter of the method.  This prevents projections that are on the order of distances associated to corrupted equations.

Under nearly
the same assumptions for the system and slightly more restrictive assumptions on the quantile parameter, we also guarantee an RK-type convergence rate for QuantileSGD(q). Our second main result is Theorem \ref{SGD_convergence}, which shows that QuantileSGD converges, again when the fraction of corruptions is sufficiently small. 

\begin{theorem} 
\label{SGD_convergence}
Let the system be defined by random matrix $\ve{A} \in \mathbb{R}^{m \times n}$ satisfying Assumptions \ref{incoherent} and \ref{bounded_density} with the constant parameters $\densconst$ and $\subgaussconst$.\textsuperscript{\ref{note:constants}}  
Then with probability at least $1-c \exp(-c_q m)$, the iterates produced by the QuantileSGD$(q - \beta)$ Method \ref{QuantileSGD} with $q \in (0,1/2)$, where in each iteration the quantile is computed using the full corrupted residual (instead of subsampling, we use $t=m$), and initialized with arbitrary point $\ve{x}_0 \in \R^n$ satisfy $$\EE\left(\norm{\ve{e}_{k}}{}^2\right) \leq \left(1-\frac {C_q}{n}\right)^k \norm{\ve{e}_{0}}{}^2$$ 
as long as the fraction of corrupted entries $\beta =|\textup{supp}(\ve{b}_C)|/m$ is a sufficiently small constant and $m\geq Cn\log n$. (Recall that $\ve{e}_k$ denotes the error vector $\ve{x}_k - \ve{x}^*.$)
\end{theorem}

In order to prove this result, we first introduce a method that we call OptSGD, which adaptively chooses an optimal step size at each iteration.  This method cannot be run in practice as it requires knowledge of $\ve{x}^*.$  However, we are able to show that QuantileSGD approximates OptSGD and therefore performs similarly well.  OptSGD may also serve as a useful benchmark when considering other SGD-type solvers for linear systems.

\rev{In each of these results, we make extensive use of theorems in high dimensional probability and do not attempt to track constants. In particular we allow our constants to depend on parameters of the underlying distributions on the rows.  We remedy this with our empirical results, which show that QuantileRK and QuantileSGD work well for practical sets of parameters.}

Finally, we consider a simpler setting that we call the \textit{streaming} setting, which may be viewed as the limiting case when the number of rows of $\ve{A}$ tends to infinity.  In this situation we do not rely on the non-asymptotic properties of random matrices and are able to give an analysis with better constants for the case when the matrix has Gaussian rows.  In particular,  Theorem \ref{streaming_sgd_convergence} shows that our methods can handle a $0.35$ fraction of corruptions, even when the values of the corruptions are chosen by an adversary. In practice, we see that the proposed methods (including the non-streaming setting) are able to accommodate much more complex cases when up to one half of the equations are corrupted.

\begin{remark}
 We get the same standard convergence rate for both methods; however, for QuantileSGD(q) we have a slightly stronger requirement on the aspect ratio of the matrix $\ve{A}$, and an additional restriction for the quantile $q < 1/2$ (whereas QuantileRK is proved for any quantile $q \in (0,1)$) In practice, QuantileSGD indeed diverges for the value of a quantile too close to one (see Figure~\ref{fig:various_quantiles_rk_sgd} (b)); however, one could safely use a much wider range of quantiles. We note that for a normalized Gaussian model (when the rows of $\ve{A}$ are sampled from the uniform distribution on the unit sphere) one can use the QuantileSGD(q) method for all $q \le 0.75$ (see Remark~\ref{gen_quant_for_uniform}). 
\end{remark}


\subsection{Related Works}\label{sec:RW}

There are many extensions and analyses of the SGD and RK methods; we review some of the results most relevant to our contributions.  The first two sections deal with consistent or \emph{noisy} systems, while the last section deals with methods for the problem of \emph{corrupted} systems.  We distinguish between \emph{corruption}, in which there are few but relatively large errors in the measurement vector, and \emph{noise}, in which there are many but relatively small errors in the measurement vector; the latter is more commonly considered within the SGD and RK literature.

\textbf{Randomized Kaczmarz variants.} The Kaczmarz method was proposed in the late 30s by Stefan Kaczmarz \cite{kaczmarzoriginal}.  The iterative method for solving consistent systems of equations was rediscovered and popularized for computed tomography as \emph{algebraic} \emph{reconstruction} \emph{technique (ART)} \cite{GBH70:Algebraic-Reconstruction, herman1993algebraic}.  While it has enjoyed research focus since that time \cite{Censor1983,Nat,  SesSta, Feinco, Hanke1990, feichtinger1995kaczmarz},
 the elegant analysis of the \emph{randomized Kaczmarz} method of \cite{strohmer2009randomized} has spurred a surge of research into variants of the Kaczmarz method. In \cite{strohmer2009randomized}, the authors proved the first exponential convergence rate in expectation \eqref{eq:RKconvrate} in the case of full-rank and consistent systems of equations.  
This result was generalized to the case when $\ve{A}$ is not full-rank in \cite{Zouzias2013}. Block methods which make use of several rows in each iteration have also become popular \cite{eggermont1981iterative, elf, popa1999block, popa2001fast, needell2014paved, rebrova2020block}.

One relevant and well-studied variant of the Kaczmarz method is that in which the row selection is performed greedily rather than randomly.
This greedy variant goes by the name \emph{Motzkin's relaxation method for linear inequalities (MM)} in the linear programming literature \cite{Motzkin1953,Agmon1954}, where convergence analyses coinciding with \eqref{eq:RKconvrate} have been shown \cite{Agmon1954}.
MM has been rediscovered in the Kaczmarz literature under the name ``most violated constraint control" or ``maximal-residual control" \cite{Censor1981,GreedyKaczmarz,Petra2016}. Several greedy extensions and hybrid randomized and greedy methods have been proposed and analyzed \cite{bai2018greedy,bai2018relaxed,DeLoera,morshed2019accelerated,loizou2019revisiting,morshed2020generalization,haddock2019greed}.
Like our methods, these greedy approaches require that sufficiently large entries of the residual be identified; however, these methods differ from ours in how these residual entries are used. 


Another relevant direction in the Kaczmarz literature are convergence analyses for systems in which the measurement matrices $\ve{A}$ have entries sampled according to a given probability distribution \cite{RefWorks:498, HN18Motzkin,haddock2019greed,rebrova2020block}.  
Our main results will make mild assumptions on the distribution of the entries of the measurement matrix.  

The convergence of many of the previously mentioned methods has been analyzed in the case that there is a small amount of noise in the system.  Generally, these analyses provide a \emph{convergence horizon} around the solution that depends upon the size of the entries of the noise.  In \cite{needell2010randomized}, the author proves that RK converges on inconsistent linear systems to a horizon which depends upon the size of the largest entry of the noise; a similar result is shown in \cite{HN18Motzkin} for MM.  In \cite{Zouzias2013, du2020pseudoinverse}, the authors develop methods that converge to the least-squares solution of a noisy system.  Meanwhile, in this work, our focus will be developing methods for systems in which there is \emph{corruption} rather than noise.  We will exploit the fact that the overdetermined system of equations has few corruptions in order to solve the uncorrupted system of equations.

\textbf{Stochastic gradient descent variants.} There has been an abundance of work developing and analyzing variants of SGD (e.g., step-size schedules, variants for specific and non-smooth objectives, etc.).  This is not meant to be a thorough survey of the literature in this area; we direct the reader to \cite{bottou2018optimization} and the references therein for a survey of recent advances, and outline here those most relevant to our approach. 

In \cite{robbins1951stochastic}, the authors provide a convergence analysis for SGD in the case that the objective is smooth and strongly convex and the step-size schedule diminishes at the appropriate rate. Such convergence results hold for fixed step-size schedules, but include a constant error term akin to the convergence horizon of RK for inconsistent systems \cite{needell2016stochastic}.  
Similar convergence rates can be proved in the case of non-smooth and non-strongly convex objectives \cite{shamir2013stochastic}; this result assumes an appropriately decreasing step-size schedule, and prove bounds on the objective value optimality gap.  Our results, unlike these, will use an iterate dependent step-size and will provide bounds on the distance between iterates and the solution of the uncorrupted system. 

Recently, batch variants that use several samples in each iteration have become popular and enjoy similar rates \cite{dekel2012optimal}.
In \cite{kawaguchi2020ordered}, the authors propose and analyze a greedy variant of SGD known as \emph{ordered SGD} that selects batches of the gradient according to the value of the associated objective components.  

An important branch of advances in the analysis of SGD deal with robustness to corruption and outliers in the objective defining data and sampled gradients, see e.g.,  \cite{chi2019median, li2020non}. Similar to our work, the aforementioned papers use quantile statistics, namely, a median-truncated SGD. 
Our methods differ from these in how we use the quantile statistic 
to achieve robustness to corruption and in our specification to linear systems.  

Here, we focus on the SGD variants developed for the LAD error; this problem is often known as LAD regression.  It has been previously noted that the $\ell_1$ objective is more robust to outliers than the $\ell_2$ objective \cite{wang2006regularized}; for this reason, there have been many algorithmic approaches to LAD regression.  These approaches have been motivated by maximum likelihood approaches \cite{li2004maximum}, rescaling techniques for low-dimensional problems \cite{bloomfield1980least}, iterative re-weighted least-squares \cite{schlossmacher1973iterative}, descent approaches \cite{wesolowsky1981new}, dimensionality reduction \cite{krvzic2018l1}, or linear programming approaches \cite{barrodale1973improved}; see \cite{gentle1988algorithms} and references therein. 

\textbf{Corrupted linear systems approaches.}
The corrupted linear system problem has been studied within the error-correction literature and has been formulated in the compressed sensing framework.  Many recovery results build upon and resemble those within the compressed sensing literature \cite{candes2005decoding}.  In particular, the optimization problem $\min \|\ve{A}\ve{x} - \ve{b}\|_0$ is a special case of the NP-hard MAX-FS problem \cite{amaldikann}.  However, if the measurement matrix $\ve{A}$ and the support of the corruption vector $\ve{b}_C$ satisfy appropriate properties, then the minimizer of the $\ell_0$ problem and the $\ell_1$ problem coincide and the problem can be solved efficiently using e.g., linear programming methods \cite{candes2005decoding, candes2005error, wright2010dense}. 

Previous work has developed and analyzed iterative methods for corrupted systems of equations.  As mentioned previously, much of the focus on this problem has been in the error correction and compressed sensing literature \cite{foucart2013math, eldar2012compressed}.  However, there has been work that has focused on iterative row-action methods; previous work in this direction includes \cite{HN18Corrupted, cloninger, amaldi}. 

Our work was inspired by \cite{haddock2018randomized, HN18Corrupted}, in which the authors propose and analyze randomized Kaczmarz variants that detect and remove corrupted equations in the system.  These methods differ from ours in that they exploit the ability of the standard RK method to detect and avoid few corruptions.  Meanwhile, our work develops variants of RK and SGD that use quantiles of the residual to converge even in the presence of corruptions. In \cite{HNRS20}, we present several methods related to those here; our results will significantly improve and generalize those in \cite{HNRS20}. 

\subsection{Notation and Definitions}\label{sec:notation}
We consider a system with measurement matrix $\ve{A} \in \mathbb{R}^{m \times n}$ and corrupted measurement vector $\ve{b} \in \mathbb{R}^{m}$ and $m \gg n$.  We denote the $i$th row of $\ve{A}$ by $\ve{a}_i$. If $\ve{A}$ is an $m\times n$ matrix and $S\subset [m],$ then let $\ve{A}_S$ denote the matrix obtained by restricting to the rows $S$.

The corrupted measurement vector $\ve{b}$ is the sum of the ideal (uncorrupted) measurement vector $\tilde{\ve{b}}$ and the corruptions $\ve{b}_C$.  The number of corruptions is a fraction $\beta \in (0,1)$ of the total number of measurements, $|\text{supp}(\ve{b}_C)| = \beta m$. Here $\text{supp}(\ve{x})$ denotes the set of indices of nonzero entries of $\ve{x}$.  The ideal measurement vector $\tilde{\ve{b}}$ defines a consistent system of equations with ideal solution $\ve{x}^*$.  We denote the $k$-th error as $\ve{e}_k := \ve{x}_k - \ve{x}^*$, where $\ve{x}_k$ denotes the $k$-th iterate of a method.

The notation $\|\ve{v}\|$ denotes the Euclidean norm of a vector $\ve{v}$. We denote the sphere in $\mathbb{R}^n$ as $S^{n-1}$, so $S^{n-1} = \left\{\ve{x} \in \mathbb{R}^n: \|\ve{x}\|=1 \right\}$. For a matrix $\ve{A}$, we denote its operator ($L_2 \to L_2$) norm by $\|\ve{A}\| = \sup_{\ve{x} \in S^{n-1}} \|\ve{A}\ve{x}\|$ and its Frobenious (or Hilbert-Schmidt) norm by $\|\ve{A}\|_F =\sqrt{\trace(\ve{A}^\top\ve{A})}$.  Throughout, we denote by $\sigma_{\min}(\ve{A})$ and $\sigma_{\max}(\ve{A})$ the smallest and largest singular values of the matrix $\ve{A}$ (that is, eigenvalues of the matrix $\sqrt{\ve{A}^\top\ve{A}}$).  Moreover, we always assume that the matrix $\ve{A}$ has full column rank, so that $\sigma_{\min}(\ve{A}) > 0$ and the convergence rate is non-trivial. \rev{We also denote the (scaled) condition number of the matrix as $\kappa(\ve{A}) = \norm{\ve{A}}{F}/\sigma_{\min}(\ve{A}) = \norm{\ve{A}}{F}\|\ve{A}^{-1}\|,$ where $\|\ve{A}^{-1}\|$ is defined to be $1/\sigma_{\min}(\ve{A}).$  }
 
Additionally, our work relies on several concepts that arise in high dimensional probability. We list all relevant definitions here, proper review of the concepts and their properties can be found in e.g., \cite{HDP}. If $X$ is a real-valued random variable, then the sub-Gaussian norm of $X$ is defined to be $\norm{X}{\Psi_2} = \inf \left\{t > 0 : \EE \exp(X^2/t^2) \leq 2 \right\}.$
If $\ve{v}$ is a random vector in $\R^n$, then the \rev{the sub-Gaussian norm of $\ve{v}$ is defined to be $\norm{\ve{v}}{\Psi_2} = \sup_{\ve{x}\in S^{n-1}} \norm{\inner{\ve{v}}{\ve{x}}}{\Psi_2}.$}
A random variable is said to be sub-Gaussian if it has finite sub-Gaussian norm.
If $\ve{v}$ is a random vector in $\R^n$ then $\ve{v}$ is said to be isotropic if $\EE(\ve{v}\ve{v}^\top) = I_n$ where $I_n$ is the identity on $\R^n.$ 

Our convergence analyses will take expectation with regards to the random sample taken in each iteration.  We denote expectation taken with regards to all iterative samples as $\mathbb{E}$.  We denote by $\EE_{k}$ the expectation with respect to the random sample selected in the $k$th iteration, conditioned on the results of the $k-1$ previous iterations of the method.

We use the following notations for the statistics of the corrupted and uncorrupted residual.  We let $\corrQ{q}{\ve{x}}$ denote the empirical $q$-quantile of the corrupted residual, 
\begin{equation}\label{corrqx}
\corrQ{q}{\ve{x}} := q\mhyphen\quantile \{\abs{\inner{\ve{a}_{i}}{\ve{x}} - b_{i}}: i \in [m]\}.
\end{equation}

\rev{For our purposes, the $q$-quantile of a multi-set $S$ is defined to be the $\lfloor q|S|\rfloor^{\text{th}}$ smallest entry of $S.$}

We let $\uncorrQ{q}{\ve{x}}$ denote the empirical $q$-quantile of the uncorrupted residual,
\begin{equation}\label{uncorrqx}
\uncorrQ{q}{\ve{x}} := q\mhyphen\quantile \left\{\abs{\inner{\ve{x} - \ve{x}^*}{\ve{a}_i}} : i \in [m]\right\}.
\end{equation}
We additionally define notation for the quantile statistics of sampled portions of the corrupted and uncorrupted residuals, 
\begin{equation}\label{subcorrqx}
\subcorrQ{q}{\ve{x}}{S} := q\mhyphen\quantile \{\abs{\inner{\ve{a}_{i}}{\ve{x}} - b_{i}}: i \in S\}
\end{equation}
and 
\begin{equation}\label{subuncorrqx}
\subuncorrQ{q}{\ve{x}}{S} := q\mhyphen\quantile \left\{\abs{\inner{\ve{x} - \ve{x}^*}{\ve{a}_i}} : i \in S\right\}
\end{equation}
where $S \subset [m]$ is the set of sampled indices.  Note that only $\corrQone{q}$ is available to us at run time since it makes use of the corrupted measurement vector $\ve{b}$; $\uncorrQone{q}$ is not available due to the use of unknown $\ve{x}^*$.  We employ $\uncorrQone{q}$ in our theoretical results as it allows us to naturally relate $\corrQone{q}$ and random matrix parameters. Finally, we let $\mean{\ve{x}}$ denote the average magnitude of the entries of $\ve{A}\ve{x}$,
\begin{equation}\label{mx}
\mean{\ve{x}} := \frac{1}{m} \sum_{i=1}^m \abs{\inner{\ve{x}}{\ve{a}_i}}.
\end{equation}


\section{Proposed Methods} \label{sec:pseudocode}

In this section, we give formal descriptions of the proposed \emph{QuantileRK(q)} and \emph{QuantileSGD(q)} methods. Our methods use the $q$-quantile entry of the residual $\abs{\ve{A}\ve{x} - \ve{b}}$ as a proxy to avoid large steps in the direction of corrupted equations. Namely, in both methods, in each iteration we sample not only an index for the RK update (which we will call the \emph{RK-index}), but also $t$ additional indices.  We then access the entries of the residual associated to these indices and compute their empirical $q$-quantile, $\subcorrQ{q}{\ve{x}}{\{i_l: l \in [t]\}}$. 

Then, the QuantileRK(q) method below takes the step (associated to the RK-index and governed by standard RK projection \eqref{eq:RKupdate}) only if the entry of the residual associated to this index is less than or equal to $\subcorrQ{q}{\ve{x}_{j-1}}{\{i_l: l \in [t]\}}$; \plzcheck{we say that a row $\ve{a}_i$ of $A$ is \emph{acceptable} on a given iteration if this is true.} We assume that the rows of our system are normalized.  If this is not the case, one could normalize the rows as they are sampled.



\begin{algorithm}
	\caption{QuantileRK(q)}\label{QuantileRK}
	\begin{algorithmic}[1]
		\Procedure{QuantileRK}{$\ve{A},\ve{b}$, q, t, N}
		\State{$\ve{x}_0 = \ve{0}$}
		\For{j = 1, \ldots, N}
		\State{sample $i_1, \ldots i_t \sim \text{Uniform} (1,\ldots, m)$}
		\State{sample $k\sim\text{Uniform}(1, \ldots, m)$}
		\If{ $\abs{\inner{\ve{a}_k}{\ve{x}_{j-1}} - b_k} \leq \subcorrQ{q}{\ve{x}_{j-1}}{\{i_l: l \in [t]\}} $}
		\State{$\ve{x}_j = \ve{x}_{j-1} - \left(\inner{\ve{x}_{j-1}}{\ve{a}_k} - b_k\right) \ve{a}_k$}
		\Else
		\State{$\ve{x}_{j} = \ve{x}_{j-1}$}
		
		\EndIf

		\EndFor{}
		
		\Return{$\ve{x}_N$}
		\EndProcedure
	\end{algorithmic}
\end{algorithm}

The QuantileSGD(q) method, Method \ref{QuantileSGD} uses the same quantile of the sampled residual $\subcorrQ{q}{\ve{x}_{j-1}}{\{i_l: l \in [t]\}}$ to define the step size.  The method steps along the direction defined by the RK update~\eqref{eq:RKupdate} based on the RK-index with step size $\gamma$ equal to $\subcorrQ{q}{\ve{x}_{j-1}}{\{i_l: l \in [t]\}}$.

\begin{algorithm}
	\caption{QuantileSGD(q)}\label{QuantileSGD}
	\begin{algorithmic}[1]
		\Procedure{QuantileSGD}{$\ve{A},\ve{b}$, q, t, N}
		\State{$\ve{x}_0 = \ve{0}$}
		\For{j = 1, \ldots, N}
		\State{sample $i_1, \ldots i_t \sim \text{Uniform} (1,\ldots, m)$}
		\State{sample $k\sim\text{Uniform}(1, \ldots, m)$}
		\State{$\gamma = \subcorrQ{q}{\ve{x}_{j-1}}{\{i_l: l \in [t]\}} $}
		\State{$\ve{x}_j = \ve{x}_{j-1} - \gamma \cdot \text{sign}\left(\inner{\ve{x}_{j-1}}{\ve{a}_k} - b_k\right)\ve{a}_k$}

		\EndFor{}
		
		\Return{$\ve{x}_N$}
		\EndProcedure
	\end{algorithmic}
\end{algorithm}

Note that this pseudocode uses only the maximum number of iterations $N$ as stopping criterion, but one could also run these methods for a specific amount of time, or implement any other stopping criterion.

Finally, we note that the behavior of both the QuantileRK and QuantileSGD depend heavily upon the input parameters.  We clarify required constraints on these parameters in the theoretical results in Section \ref{sec:theor_results}.  Additionally, we discuss the effect of these parameter choices on computation and other implementation considerations in Section \ref{sec:imp}.

\section{Theoretical Results}\label{sec:theor_results}
Here we state and prove our theoretical results.  We begin with foundational results in high-dimensional probability in Subsection~\ref{subsec:foundations} and then prove our main convergence results, Theorems~\ref{median_convergence} and \ref{SGD_convergence}, in Subsections~\ref{subsec:QuantileRK} and \ref{subsec:QuantileSGD}.  In our proof of convergence of QuantileSGD$(q)$, Theorem~\ref{SGD_convergence}, we propose an ideal method, OptSGD, and demonstrate that it is well approximated by QuantileSGD$(q)$.  We additionally prove convergence of QuantileSGD$(q)$ in the simpler streaming setting in Subsection~\ref{sec:streaming}.

\subsection{Theoretical Foundations}\label{subsec:foundations}
In this subsection, we prove several fundamental results which we apply in our convergence analyses for QuantileRK and QuantileSGD in Sections \ref{subsec:QuantileRK} and \ref{subsec:QuantileSGD}. 

\subsubsection{Auxiliary results -- properties of random matrices}
 For the largest singular values of a random matrix with independent isotropic rows, we will be using the following standard bound (the proof can be found in e.g., \cite[Theorem  4.6.1]{HDP}).
\begin{theorem}\label{norm_bound}
Let $\ve{A} \in \mathbb{R}^{m\times n}$ be a matrix whose rows are independent, mean zero, sub-Gaussian and isotropic with sub-Gaussian norm bounded by $\subgaussconst.$  Then the largest singular value (operator norm) of $\ve{A}$ is bounded by \[\sqrt{m} + C\subgaussconst^2(\sqrt{n}+t)\] with probability at least $1 - 2\exp(-t^2).$
\end{theorem}

The smallest singular value of random matrices is sometimes called a ``hard edge" as it is typically harder to quantify. This is the case in our application as well; we will prove Proposition~\ref{sub_conditioning} giving a uniform lower bound on the singular values \plzcheck{of} the submatrices of $\ve{A}$.

The first ingredient that we need for this (and it will be used in other places later in the text as well) is an \emph{$\eps$-net for the unit sphere}.  We say that $\mathcal{N}$ is an $\eps$-net of a set $S\subseteq \R^n$ if $\mathcal{N}$ is a subset of $S$ and each point in $S$ is within a Euclidean distance $\eps$ of some point in $\mathcal{N}.$  The $\eps$-covering number of $S$ is the cardinality of the smallest $\eps$-net for $S$. We will use the fact that the $\eps$-covering number of $S^{n-1}$ is bounded by $(3/\eps)^n$ \cite[Corollary 4.2.13]{HDP}.

We will also use the following direct corollary of the Hoeffding's inequality (see, e.g., \cite[Theorem 2.6.2]{HDP}) that subgaussian random variables concentrate as well as Gaussians under taking means.

\begin{lemma}
\label{mean_subgaussian}
Let $X_1,\ldots, X_m$ be i.i.d.\ subgaussian random variables with subgaussian norm $\subgaussconst.$ Then the subgaussian norm of the mean satisfies \[\norm{\frac{1}{m}\sum_{i=1}^m X_i}{\Psi_2} \leq C \frac{\subgaussconst}{\sqrt{m}}.\]
\end{lemma}

Next, the following anti-concentration lemma for random vectors with bounded density is a direct corollary of \cite[Theorem~1.2]{rudelson2015small}.

\begin{lemma}
\label{bounded_pdf}
Let $\ve{x}$ be a random vector in $\R^n$ such that the density function of each coordinate $x_i$ is bounded by $\densconst \sqrt{n}$, where $\densconst > 0$ is an absolute constant.  Then for any fixed $\ve{u}\in S^{n-1}$ we have  $$\Pr\left(\abs{\inner{\ve{x}}{\ve{u}}} \leq \frac{\sqrt{t}}{\sqrt{n}}\right) \leq 2\sqrt{2} \densconst\sqrt{t}.$$
\end{lemma}




We will use this anti-concentration result to prove a uniform lower bound for the smallest singular value over all $\alpha m \times n$ submatrices of a tall random matrix of the size $m \times n$. It is well known that for a single fixed (row-)submatrix $\ve{A}_T$ of that size, $\sigma_{\min}(\ve{A}_T) \gtrsim \sqrt{m}/\sqrt{n}$ (see e.g., \cite[Theorem 4.6.1]{HDP}). However, naively taking a union bound over all $\binom{m}{\alpha m}$ $\alpha m$-tall row submatrices results in a trivial probability bound. In Proposition~\ref{sub_conditioning}, we provide a more delicate row-wise analysis by employing Chernoff's bound to provide a good uniform lower bound with probability exponentially close to one.

\begin{proposition}
\label{sub_conditioning}
Let $\alpha \in (0,1]$ and let random matrix $\ve{A} \in \mathbb{R}^{m\times n}$ satisfy Assumptions \ref{incoherent} and \ref{bounded_density}.   Then there exist absolute constants $C_1, c_2 > 0$ so that if the matrix $\ve{A}$ is tall enough, namely, 
\begin{equation}\label{size_assump}\frac mn > C_1 \frac{1}{\alpha} \log \frac{\densconst \subgaussconst}{\alpha},
\end{equation}
then the following uniform lower bound holds for the smallest singular values of all its row submatrices that have at least $\alpha m$ rows.
$$
\Pr\left(\inf_{\substack{ T\subseteq [m]: \\ \abs{T} \geq \alpha m}} \sigma_{\min}(\ve{A}_T) \geq \frac{\alpha^{3/2}}{24\densconst} \sqrt{\frac mn}\right) \ge 1 -3\exp(-c_2\alpha m)
$$
\end{proposition}

\begin{proof}
 Let $\eps \in (0, 1]$ be a constant (chosen below in \eqref{eps_choice}). Recall that there is  an $\eps$-net $\mathcal{N}$ of $S^{n-1}$ with the cardinality $\abs{\mathcal{N}} \leq \parens{\frac{3}{\eps}}^n$. That is, for any $\ve{y} \in S^{n-1}$ there exists $\ve{x} \in \mathcal{N}$ such that $\|\ve{y} - \ve{x}\|_2 \le \eps$. Taking the infimum over all unit norm vectors $\ve{x}$, we get that for any $T \subseteq [n]$, we have 
 \begin{equation}\label{inf_net_bound}
 \sigma_{\min}(\ve{A}_T) = \inf_{\ve{y} \in S^{n-1}}\|\ve{A}_T \ve{y}\| \geq \parens{\inf_{\ve{x}\in \mathcal{N}}\norm{\ve{A}_T\, \ve{x}}{}} - \eps\norm{\ve{A}_T}.
 \end{equation} 
We will bound \plzcheck{the} two terms in the right hand side of \eqref{inf_net_bound} separately. First, for any subset $T \subset [n]$, we can bound $\norm{\ve{A}_T}{\text{op}} \le \norm{\ve{A}}{\text{op}}$, and so by Theorem \ref{norm_bound}
 $$
 \Pr\left(\norm{\ve{A}_T}{} \le (1+C\subgaussconst^2)\sqrt{\frac mn}\right)\ge 1 - 2\exp(-cm)
 $$  
for some absolute constants $C, c > 0$. Let us choose
\begin{equation}\label{eps_choice}
\eps = \frac{\alpha^{3/2}}{24\densconst(1+C\subgaussconst^2)}.
\end{equation}
To bound the first term in the right-hand side of \eqref{inf_net_bound}, first consider a fixed $\ve{x}$ in $\mathcal{N}.$  For $i\in[m],$ let $\mathcal{E}^{\ve{x}}_i$ be the event 
$$ \mathcal{E}^{\ve{x}}_i :=\left\{\abs{\inner{\ve{a}_i}{\ve{x}}}^2 < \frac{\alpha^2}{64 \densconst^2}\cdot \frac1n\right\}.$$  
By Lemma \ref{bounded_pdf}, $\Pr(\mathcal{E}^{\ve{x}}_i)\leq \alpha/4$ for each fixed $\ve{x} \in S^{n-1}$ and $i \in [m]$.  A Chernoff bound then gives 
$$\Pr\left(\; \text{events } \mathcal{E}^{\ve{x}}_i \text{ occur for at least } \alpha m/2 \text{ indices } i \in [m]\; \right) \leq \exp(-\alpha m / 12).$$ 
Now, for any fixed $\ve{x}$, provided that $\mathcal{E}_i^{\ve{x}}$ occurs for at most $\alpha m/2$ indices $i \in [m]$, for all $T$ with $|T| \geq \alpha m$ we have 
$$\norm{\ve{A}_T\, \ve{x}}{} \geq \sqrt{\parens{\frac{\alpha}{2}m}\cdot\parens{\frac{\alpha^2}{64\densconst^2} \cdot\frac1n} } = \frac{\alpha^{3/2}}{12 \densconst} \sqrt{\frac{m}{n}}.$$   
Finally, taking a union bound over $\ve{x} \in \mathcal{N}$, we have $$\Pr\left(\inf_{\ve{x}\in \mathcal{N}}\norm{\ve{A}_T\, \ve{x}}{} \leq C_\alpha \sqrt{\frac mn}\right) \leq \exp\left(n\log\frac{3}{\eps} - \frac{\alpha m}{12}\right) \le \exp\left(-\frac{\alpha m}{24}\right),$$ where the last inequality holds due to the submatrix size assumption \eqref{size_assump} and our choice of $\eps$ in \eqref{eps_choice}.

Returning to the estimate \eqref{inf_net_bound}, we can now conclude that with probability 
$$1 - 2\exp(-cm) -\exp(-\alpha m/24) \ge 1 - 3 \exp(-c_2 \alpha m),$$ for all $T$ with $\abs{T} \geq \alpha m,$ $$\sigma_{\min}(\ve{A}_T) \geq \frac{\alpha^{3/2}}{12 \densconst} \sqrt{\frac mn} - \eps \cdot (1+C\subgaussconst^2) \sqrt{\frac mn} \ge \frac{\alpha^{3/2}}{24 \densconst}\sqrt{\frac mn}
$$
due to our choice of $\eps$ in \eqref{eps_choice}. This concludes the proof of Proposition~\ref{sub_conditioning}.
\end{proof}

\begin{remark}\label{bdd_density_assumption}
Note that the bounded density assumption is crucial for Proposition~\ref{sub_conditioning}. For instance, the rows of a normalized Bernoulli matrix violate the hypotheses of Lemma \ref{bounded_pdf}, and Proposition \ref{sub_conditioning} does not apply.   Unfortunately this cannot be overcome.  Indeed, consider taking $\ve{x}$ to be the vector $(1,-1,0,\ldots, 0).$  Then $\inner{\ve{a}_i}{\ve{x}} = 0$ with probability $1/2.$  So if $\alpha < 1/2$ in Proposition~\ref{sub_conditioning} then $\ve{x}$ will lie in the kernel of some $\alpha m \times n$ submatrix of $\ve{A}$ with high probability, violating the uniform lower bound on the smallest singular value of the submatrices.
\end{remark}

\subsubsection{Auxiliary results -- structure of the residual}
Recall that $\ve{a}_i$ denotes a (normalized) row of the matrix $\ve{A}$. We recall the notations for the statistics of the corrupted and uncorrupted residuals; we denote the $q$-quantile of the corrupted residual as $\corrQ{q}{\ve{x}}$, and the $q$-quantile of the uncorrupted residual as $\uncorrQ{q}{\ve{x}}$.  We additionally recall that the empirical mean of the entries of $\ve{A}\ve{x}$ is denoted $\mean{\ve{x}}$. 

The key observation is that for all uncorrupted indices $i$ we have $$ \inner{\ve{x}_k - \ve{x}^*}{\ve{a}_i} = \inner{\ve{x}_k}{\ve{a}_i} - \inner{\ve{x}^*}{\ve{a}_i} = \inner{\ve{x}_k}{\ve{a}_i} - b_i.$$  Each of $\ve{x}_k, \ve{a}_i,$ and $b_i$ is available at runtime (unlike the exact solution $\ve{x}^*$), so this quantity may be computed directly. Then, due to the robustness to noise of the order statistics, we can use the quantiles of the corrupted residual, $\corrQ{q}{\ve{x}_k}$, to estimate quantiles of the uncorrupted residual, $\uncorrQ{q}{\ve{x}_k}$.

In particular, the following straightforward implication of the definition of quantiles is used in the proof of QuantileSGD convergence. We omit the proof.
\begin{lemma}\label{residual_vs_emp_quantiles}
With at most a $\beta$ fraction of samples corrupted by an adversary, we have $$\uncorrQ{q-\beta}{\ve{x}_k} \leq \corrQ{q}{\ve{x}_k} \leq \uncorrQ{q+\beta}{\ve{x}_k}.$$
\end{lemma}

We will estimate empirical uncorrupted quantiles $\uncorrQ{q}{\ve{x}}$ instead of $\corrQ{q}{\ve{x}}$ first. The rest of this section consists of two parts: upper bounds for $\uncorrQ{q}{\ve{x}}$, and  lower bounds for $\uncorrQ{q}{\ve{x}}$. As in the previous subsection, the main challenge is to get uniform high-probability estimates over the unit sphere.

\subsubsection{Upper bounds for the empirical quantiles} 

The next lemma shows that any fairly large collection of rows is reasonably incoherent.  We will need this result in order to handle situations in which the locations of corruptions are chosen adversarially.

\begin{lemma}
\label{avg_bound}
Let random matrix $\ve{A} \in \mathbb{R}^{m \times n}$ satisfy Assumption \ref{incoherent}.  With probability at least $1 - 2\exp(-cm)$  we have that for all unit vectors $\ve{x}\in\R^n$ and every $T\subseteq [m]$,
$$\sum_{i\in T}\abs{\inner{\ve{x}}{\ve{a}_i}} \leq C_{\subgaussconst}\sqrt{\frac{m|T|}{n}}.$$
\end{lemma}

\begin{proof} Consider a vector $\ve{s} = (s_i)\in \{-1,0,1\}^m$ defined by
\[
s_i = 
\begin{cases}
\sgn(\langle \ve{x}, \ve{a}_i \rangle), &\text{ if } i \in T \\
0, &\text{ otherwise, }
\end{cases}
\]
for $i \in [m]$. Note that $\|\ve{s}\| \leq \sqrt{|T|}$.

The left hand side of the desired inequality can be written as 
$$\sum_{i\in T}\abs{\inner{\ve{x}}{\ve{a}_i}} = \sum_{i=1}^m \inner{\ve{x}}{s_i \ve{a}_i} = \inner{\ve{x}}{\sum_{i\in [m]} s_i \ve{a}_i}\leq \norm{\sum_{i\in [m]} s_i\ve{a}_i}{} = \norm{\ve{A}^\top\ve{s}}{}.$$
Now the last norm can be estimated using the bound from Theorem~\ref{norm_bound} (since the $\sqrt{n}$-rescaled rows of $\ve{A}$ are isotropic and bounded) to get
$$\norm{\ve{A}^\top\ve{s}}{} \le \norm{\ve{A}^\top}{} \norm{\ve{s}}{} = \norm{\ve{A}}{} \norm{\ve{s}}{} \le C_{\subgaussconst}\sqrt{\frac{m|T|}{n}}.$$
This concludes the proof of Lemma~\ref{avg_bound}.
\end{proof}


The next proposition allows us to bound the quantiles computed by QuantileRK and QuantileSGD . We assume that $\alpha m$ ``bad" indices from the next lemma are those that will be excluded by the quantile statistic.


\begin{proposition} 
\label{med_bound}
Let $\alpha \in (0,1]$ and let random matrix $\ve{A} \in \mathbb{R}^{m\times n}$ with $m\geq n$ satisfy Assumption \ref{incoherent}.  Then for $t\geq 0$ and with probability $1-2\exp(-t^2m)$,  for every $\ve{x} \in S^{n-1}$ we have the bound \begin{equation} 
\label{eq:M_bound}
M(\ve{x}) \le \frac{1}{\sqrt{n}} + \subgaussconst\parens{\frac{c_1}{\sqrt{m}} + \frac{c_2 t}{\sqrt{n}}}.
\end{equation}

As a consequence, with probability $1-2\exp(-m),$ for every $\ve{x}$ in $S^{n-1}$ the bound $$\abs{\inner{\ve{a}_i}{\ve{x}}} \leq \frac{1}{\alpha} \frac{C\subgaussconst}{\sqrt{n}}$$ holds for all but at most $\alpha m$ indices $i.$
\end{proposition}

\begin{proof} 
We will use a chaining argument to show that the averages  $M(\ve{u})$ (defined by~\eqref{mx}) are concentrated uniformly over the sphere. For $\ve{u},\ve{v}\in S^{n-1}$, we have $$\abs{M(\ve{u}) - M(\ve{v})} \leq \frac1m \sum_{i=1}^m \abs{\inner{\ve{a}_i}{\ve{u}-\ve{v}}}.$$  The terms in this sum are independent sub-Gaussian random variables with sub-Gaussian norm no larger than $\subgaussconst\norm{\ve{u}-\ve{v}}{}/\sqrt{n}$.  Therefore by Lemma \ref{mean_subgaussian},
\begin{align*}
\norm{M(\ve{u}) - M(\ve{v})}{\psi_2} \le \frac{C \cdot \subgaussconst\norm{\ve{u}-\ve{v}}{}}{\sqrt{m}\sqrt{n}}.
\end{align*}

By the tail bound version of Dudley's inequality (\cite[Theorem 8.1.6]{HDP}) and the bound $\parens{3/\eps}^n$ for the $\eps$-covering number of the unit sphere, we then have with probability at least $1- 2\exp(-t^2m)$ 
\begin{equation}\label{difference_xu}
\sup_{\ve{u},\ve{v}\in S^{n-1}} \abs{M(\ve{u}) - M(\ve{v})} \le \frac{C_1\subgaussconst}{\sqrt{m} \sqrt{n}} \left(\sqrt{n} + \text{diam}\parens{S^{n-1}}t\sqrt{m}\right) = \subgaussconst\parens{\frac{C_1}{\sqrt{m}} + \frac{2C_1 t}{\sqrt{n}}}.
\end{equation}

Next, we claim that with probability one $M(\ve{u})$ is bounded by $n^{-1/2}$ for some $\ve{u}$.  This follows from the probabilistic method. \rev{Indeed, since $\ve{a}_i$ has unit norm, for $\ve{u}$ sampled uniformly over the unit sphere we have $\EE |\inner{\ve{u}}{\ve{a}_i}| = \EE |\inner{\ve{u}}{\ve{a}_1}|$ for all $i$ by symmetry of $\ve{u}$ (note that the expectation above is taken over $\ve{u}$).}  Therefore \begin{equation}\label{individual_xu}
\left(\EE M(\ve{u})\right)^2 = \left(\EE\abs{\inner{\ve{u}}{\ve{a}_1}}\right)^2 \leq \EE \abs{\inner{\ve{u}}{\ve{a}_1}}^2 = n^{-1}.
\end{equation}
 For some $\ve{u}$, $M(\ve{u})$ is at most its expectation and hence $M(\ve{u}) \leq n^{-1/2}$ for some $\ve{u} \in S^{n-1}$.  


The estimates \eqref{individual_xu} and \eqref{difference_xu} together imply that \eqref{eq:M_bound} holds with probability at least $1-2\exp(-t^2m).$


For the second half of Proposition \ref{med_bound} we specialize to the case $t=1$ and use that $m\geq n$ to get the bound $M(\ve{x}) \leq C\subgaussconst/\sqrt{n}$ with probability at least $1-2\exp(-m).$

Now, note that if for some $\ve{u} \in S^{n-1}$ more than $\alpha m$ members of the sum defining $M(\ve{u})$ are bigger than $C \subgaussconst/\alpha \sqrt{n}$, then
$$|M(\ve{u})| > (\alpha m)\frac{1}{m} \frac{C\subgaussconst}{\alpha \sqrt{n}} = \frac{C\subgaussconst}{\sqrt{n}},$$
contradicting \eqref{eq:M_bound}. This concludes the proof of the Proposition~\ref{med_bound}.
\end{proof}

We will use Proposition \ref{med_bound} in a slightly more general, although equivalent form.
\begin{corollary}
\label{med_cor}
Let $\alpha \in (0,1]$, let random matrix $\ve{A} \in \mathbb{R}^{m\times n}$ satisfy Assumption \ref{incoherent}, and suppose that $\ve{A}\ve{x}^* = \ve{b}$ for some $\ve{x}^*\in \R^n$ (that is, the linear system is uncorrupted).  Assuming that $m\geq n$, there exists a constant $C_{\subgaussconst} > 0$ so that with probability at least $1-2\exp(-m)$,  for every $\ve{x} \in \R^{n}$ the bound $$\abs{\inner{\ve{a}_i}{\ve{x}}- b_i} \le \frac{C_{\subgaussconst}}{\alpha\sqrt{n}}\norm{\ve{x}-\ve{x}^*}{}$$ holds for all but at most $\alpha m$ indices $i.$
\end{corollary}
\begin{proof}
To prove this, one simply writes $b_i = \inner{\ve{a}_i}{\ve{x}^*}$, and applies Proposition \ref{med_bound} for a unit vector $(\ve{x} - \ve{x}^*)/\|\ve{x} - \ve{x}^*\|$.
\end{proof} 
\begin{remark}\label{corr_med_cor}
Note that for the system corrupted by $\beta m$ corruptions, analogously to Corollary~\ref{med_cor}, we have that with probability at least $1-2\exp(-m)$,  for every $\ve{x} \in \R^{n}$ the bound $$\abs{\inner{\ve{a}_i}{\ve{x}}- b_i} \le \frac{C_{\subgaussconst}}{\alpha\sqrt{n}}\norm{\ve{x}-\ve{x}^*}{}$$ holds for all but at most $(\alpha + \beta) m$ indices $i.$
\end{remark}

\rev{Our analysis of QuantileSGD will also require the following lower bound on $M(\ve{x}).$}
\begin{proposition}
\label{M_lower_bound}
\rev{
Let $\ve{A}\in \R^{m\times n}$ satisfy Assumption \ref{incoherent} \plzcheck{and $\ve{x} \in S^{n-1}$}. For $m\geq C_K n$ we have $M(\ve{x}) \geq \frac{c}{\sqrt{n}}$ with probability at least $1-2\exp(-c_K m).$}
\end{proposition}

\begin{proof}
\rev{
Parallel to the analysis following \eqref{individual_xu} , some $\ve{u}\in S^{n-1}$ is at least the expectation $\EE M(\ve{u}) = \EE \abs{X_n}$ where $X_n := \inner{\ve{a}_1}{\ve{u}}$ and $\ve{u}$ is uniform over $S^{n-1}$ (as in the proof of Proposition \ref{med_bound}, the expectation is again taken over $\ve{u}$, and is independent of the $\ve{a}_i$'s by symmetry of $\ve{u}.$)   By The Projective Central Limit Theorem \rev{(see for instance Remark 3.4.8 in \cite{HDP})}, $\sqrt{n}X_n$ converges in distribution to a standard normal as $n\rightarrow\infty.$  Moreover the random variables $\sqrt{n} X_n$ are uniformly integrable and so $\EE(\abs{\sqrt{n}X_n}) \rightarrow \mu$ where $\mu\approx 0.78$ is the mean of a standard half-normal random variable.\footnote{i.e., the absolute value of a standard normal random variable} In particular $\EE(\abs{\sqrt{n}X_n})$ is bounded below by a constant uniformly in $n.$  By the argument in Proposition \ref{med_bound} we then have that for all $\ve{x}\in S^{n-1},$ $$M(\ve{x}) \geq \frac{c}{\sqrt{n}} - \subgaussconst\parens{\frac{c_1}{\sqrt{m}} + \frac{c_2 t}{\sqrt{n}}}$$ with probability at least $1-2\exp(-t^2 m).$ Provided that $m\geq C_{\subgaussconst} n$ this bound simplifies to $$M(\ve{x}) \geq \frac{c}{\sqrt{n}}$$ with probability at least $1 - 2\exp(-c_{\subgaussconst} m)$ as desired.}
\end{proof}

\begin{remark}
\label{M_concentration}
If in addition, one assumes that $n$ is a sufficiently large constant $C_{\epsilon}$ so that $\EE(\abs{\sqrt{n}X_n})$ is near $\mu$, then one can have $$M(\ve{x}) \geq \frac{\mu - \epsilon}{\sqrt{n}}$$ with probability at least $1-2\exp(-c_{\subgaussconst,\epsilon} m)$ provided that $m\geq C_{\subgaussconst,\epsilon} n.$ The latter bound allows us to extend the guarantees for the QuantileSGD algorithm for a wider range of quantiles, under additional restrictions on the model (we will not carry out this analysis in detail, however see Remark~\ref{gen_quant_for_uniform}).

Of course each of these facts applies equally well to the uncentered and rescaled setting of Corollary \ref{med_cor}.
\end{remark}

\subsubsection{Lower bound for empirical quantiles}
We also use the following lower-bound variant of the above result when analyzing QuantileSGD.

\begin{lemma} 
\label{quantile_lower}
Let $\ve{A} \in \mathbb{R}^{m \times n}$ be a random matrix satisfying Assumption \ref{bounded_density} and let $\ve{x}^*$ and $\ve{b}$ be such that $\ve{A}\ve{x}^* = \ve{b}$ (that is, the linear system is consistent). 
If $m\geq \frac{C}{q} n\log\parens{\frac{cn\densconst}{q}}$ then
\[\Pr\left\{\corrQ{q}{\ve{x}} \geq \frac{cq}{\densconst\sqrt{n}}\norm{\ve{x}-\ve{x}^*}{} \text{ for all } \ve{x} \in \R^n \right\} \ge 1-\exp(-cm).\]
\end{lemma}

\begin{proof}
We assume without loss of generality that $\ve{x}^* = \ve{0}$  and $\ve{b} = \ve{0}.$ The general result follows in the same way as in Corollary~\ref{med_cor} above.  By scaling, it suffices to prove the result for $\ve{x}\in S^{n-1}.$

First consider a fixed $\ve{x}$ in $S^{n-1}.$ By Lemma \ref{bounded_pdf}, we can choose $c_{q} = cq/\densconst$ so that, \begin{equation} \label{pre_chernoff} \Pr\parens{\abs{\inner{\ve{x}}{\ve{a}_i}}\leq \frac{2c_q}{\sqrt{n}}} \leq \frac{q}{2}
\end{equation}
for all $i$. By a Chernoff bound, $\corrQ{q}{\ve{x}} \geq 2c_q / \sqrt{n}$ with probability at least $1-\exp(-qm/6).$

Let $\mathcal{N}$ be a $c_q/\sqrt{n}$-net of $S^{n-1}$ which we can take to have size \[\abs{\mathcal{N}} = \parens{\frac{3}{c_q/\sqrt{n}}}^n = \exp(n\log(3\sqrt{n}/c_q)).\]  By a union bound, there are constants so that if \[m\geq \frac{C}{q} n\log\parens{\frac{cn\densconst}{q}},\] then the quantile bound \ref{pre_chernoff} holds for all $x$ in $\mathcal{N}$ with probability at least $1 - \exp(-c m)$.

In order to upgrade our bound on $\mathcal{N}$ to all of $S^{n-1},$ it remains to show that $\corrQ{q}{\ve{x}}$ is stable under small perturbations of $\ve{x}$.

Suppose that $\ve{x}$ and $\ve{y}$ in $\R^n$ are arbitrary.  Then for all $i$, we have the bound \[\abs{\inner{\ve{x}}{\ve{a}_i}} - \abs{\inner{\ve{y}}{\ve{a}_i}} \leq \abs{\inner{\ve{x}}{\ve{a}_i} - \inner{\ve{y}}{\ve{a}_i}} = \abs{\inner{\ve{x}-\ve{y}}{\ve{a}_i}} \leq \norm{\ve{x}-\ve{y}}{}.\] 

Therefore \[\abs{\inner{\ve{x}}{\ve{a}_i}}-\norm{\ve{x}-\ve{y}}{} \leq \abs{\inner{\ve{y}}{\ve{a}_i}} \leq \abs{\inner{\ve{x}}{\ve{a}_i}}+\norm{\ve{x}-\ve{y}}{}.\] 
By taking the $q$-quantiles over $i$ and using monotonicity of quantiles, it follows that
\begin{equation}\label{quant}
\abs{\corrQ{q}{\ve{x}} - \corrQ{q}{\ve{y}}} \leq\norm{\ve{x}-\ve{y}}{}
\end{equation}
for all $\ve{x},\ve{y}.$

Each point in $S^{n-1}$ is within $c_q/\sqrt{n}$ of some point in $\mathcal{N}.$ Lemma \ref{quantile_lower} follows by combining \eqref{quant} with our bound on $\corrQ{q}{\ve{x}}$ over $\mathcal{N}.$
\end{proof}

\begin{remark}
We require the aspect ratio of $\ve{A}$ to be at least order $\log(n).$ It is plausible that Lemma~\ref{quantile_lower} can be improved to hold for constant aspect ratios as was the case for Proposition \ref{med_bound}.  We will not attempt to do so, and as a result we require QuantileSGD to have a slightly stronger condition on the aspect ratio of $\ve{A}$ than QuantileRK.
\end{remark}




\subsection{Analysis of the QuantileRK method} \label{subsec:QuantileRK}

In this section we provide a proof that the QuantileRK method converges.  

{\bfseries Roadmap.} The proof will proceed as follows.  We condition on the sampling of a row that will be accepted by the QuantileRK iteration\plzcheck{; recall a row is acceptable in a given iteration if the entry of the residual associated to this row is less than or equal to $\subcorrQ{q}{\ve{x}_{j-1}}{\{i_l: l \in [t]\}}$}.  We then show that the uncorrupted rows help substantially, while the corrupted rows do not overly affect the convergence. Conditioned on the current row being uncorrupted, we argue that an iteration of the QuantileRK method brings us closer in expectation to $\ve{x}^*.$  To accomplish this, we show that the restriction of $\ve{A}$ to the acceptable uncorrupted rows is well-conditioned via Lemma \ref{avg_bound}.  In that case, the current iteration of QuantileRK is equivalent to an iteration of the standard RK method on the restricted matrix.  This allows us to apply a known per-iteration guarantee for RK.

To argue that corrupted rows do not significantly harm convergence, we consider a subset $J \subset [m]$, of row indices with $|J|/m \ge c$, and with $J$ containing all corrupted indices as a subset. By making $J$ sufficiently large, we ensure that the subset of the rows of $\ve{A}$ indexed by $J$ inherits incoherence properties from the full matrix (uniformly over all such subsets, due to Proposition~\ref{sub_conditioning}). Incoherence will ensure that the average projection of $\ve{x}$ onto a corrupted hyperplane moves the point in a direction nearly orthogonal to $\ve{x}-\ve{x}^*.$  The length of such a step is bounded by $c/\sqrt{n}$ by Proposition \ref{med_bound}, so a ``bad" step is unlikely to move $\ve{x}$ much further from $\ve{x}^*.$ In particular, a constant number of ``good" steps will suffice to ``cancel out" a bad step.  If the fraction of bad rows is sufficiently small, then the QuantileRK method will enjoy linear convergence to $\ve{x}^*.$

\begin{proof}[Proof of Theorem~\ref{median_convergence}]


We will start by introducing some useful notation.  Recall that each instance of an absolute constant may refer to a different constant value; however, we track the dependence of $q$ on $\beta$ explicitly.

Let $\mathcal{E}_{\Accept}(k)$ denote the event that we sample an acceptable row at the $k$-th step of the method; that is, if the if-statement in line 6 of the QuantileRK Method \ref{QuantileRK} evaluates to true for that row. Recall that an $i$-th row of $\ve{A}$ is acceptable at iteration $k$ if $\abs{\inner{\ve{x}_k}{\ve{a}_i} - b_i} \le \corrQ{q}{\ve{x}_k}$, where $\corrQone{q}$ is defined as in \eqref{corrqx}.  \rev{Clearly, $\Pr(\mathcal{E}_{\Accept}(k)) = \lfloor qm \rfloor/m$ for any integer $k \ge 1$.}

Further, we will consider three subsets of indices denoted as $J$, $I_1$ and $I_2$. Let $J$ denote a collection of indices of size\footnote{We assume without loss of generality that $\beta m$ is an integer. If this is not the case, consider $\beta'$ such that $\beta'm = \lceil \beta m\rceil$ instead of $\beta$ throughout the proof.} $2\beta m$ which contains all corrupted indices and at least $\beta m$ acceptable indices. We assume that $\beta < q$ so there exists that many acceptable indices (as there are exactly $\lfloor qm\rfloor$ acceptable indices total).  Then, all acceptable indices are split into two types: those inside the set $J$ (we denote them $I_1$, by construction, $|I_1| \ge \beta m$) and those outside of $J$ (we denote them $I_2$). Finally, let $\mathcal{E}^k_L$ denote the event that $k$-th iteration of the QuantileRK method samples an index from an index subset $L \subset [n]$.

We first observe that \rev{
\begin{equation}\label{cond_to_accept}
\EE_k(\norm{\ve{e}_{k+1}}{}^2) = 
(\lfloor qm \rfloor/m) \EE_k(\norm{\ve{e}_{k+1}}{}^2|\mathcal{E}_{\Accept}(k+1)) + (1 - \lfloor qm \rfloor/m) \norm{\ve{e}_{k}}{}^2,
\end{equation}}
since QuantileRK(q) does not update $x_k$ if a sampled row index was not acceptable.

Conditioned on choosing an acceptable row, we either pick an index from $I_1$ or from $I_2$, and the conditional probability $p_J$ to choose an index in $I_1$ satisfies $p_J \le 2\beta m /q m = 2 \beta/q$ (the upper bound refers to the case when $I_1 = J$).

Now, given $\mathcal{E}^{k+1}_{I_2}$, the iterate $\ve{x}_{k+1}$ is obtained by applying an iteration of the Standard RK method for the matrix $\ve{A}_{I_2}$. 
Note that $I_2$ \rev{has size at least $(q - 2\beta)m - 1$, since at least \plzcheck{$\lfloor qm \rfloor$} indices are acceptable, and at most $2\beta m$ of these are contained in $I_1.$} Next we apply Proposition \ref{sub_conditioning} with $\alpha = q - 2\beta - \frac{1}{m} > 0.$ \rev{As long as $\beta$ is at most a constant factor of $q$ (e.g., $\beta \leq q/4$), then we have $\alpha \geq cq.$ The proposition then gives  $\|\ve{A}_{I_2}^{-1}\|_2 \leq C_{\beta, \densconst}\sqrt{n/m}$ with probability $1 - 3\exp(-c_{q}m)$ provided that $$\frac{m}{n} \ge C_{q, \densconst} := C \frac{1}{\alpha} \log\parens{\frac{\densconst \subgaussconst}{\alpha}}.$$}  \rev{Since all the rows of $\ve{A}$ are normalized to have unit norm, we also know that $\norm{\ve{A}}{F} = \sqrt{m}$. Therefore, with high probability, we may bound the condition number of $\ve{A}_{I_2}$ as 
\begin{equation}\label{cond_number}
\kappa(\ve{A}_{I_2}) \leq \|\ve{A}\|_F \|\ve{A}_{I_2}^{-1}\|_2 \leq \sqrt{m} C_{q, \densconst} \sqrt{n/m} = C_{q, \densconst} \sqrt{n}.
\end{equation}}
Note that Proposition~\ref{sub_conditioning} gives a uniform \plzcheck{lower} bound for the condition number for all index subsets of size at least $\alpha m$. So in each iteration of the method, $\ve{A}_{I_2}$ will have a good condition number \plzcheck{upper} bounded by \eqref{cond_number} with probability at least \rev{$1 - 3\exp(-c_{q}m)$}. \rev{In particular, $\ve{A}_{I_2}$ has full rank since $\|{\ve{A}_{I_2}^{-1}}\|$ is finite.} Then, by the analysis of the Standard RK method \cite{strohmer2009randomized} given in \eqref{eq:RKconvrate}, we have 
$$\EE_k(\norm{\ve{e}_{k+1}}{}^2|\mathcal{E}^{k+1}_{I_2}) \leq \left(1- \frac {c_1}{n}\right)\norm{\ve{e}_k}{}^2.$$


Now, we consider two cases. In the no corruptions case when $\beta = 0$, we have that the set $I_1$ is empty and $p_J = 0$ by definition. So, 
\begin{equation}\label{beta_zero}
\EE_k(\norm{\ve{e}_{k+1}}{}^2) \le
q \left( 1 - \frac{c_1}{n} \right) \norm{\ve{e}_k}{}^2 + (1 - q) \norm{\ve{e}_{k}}{}^2 \le \left( 1 - \frac{qc_1}{n} \right) \norm{\ve{e}_k}{}^2.
\end{equation}

In the other case, when $\beta > 0$, we need to consider the second possibility, if the next index was coming from $I_1$. 
Conditioned on taking an acceptable row, we can choose $h_i$ with $\abs{h_i} \leq \corrQ{q}{\ve{x}_k}$, so that
\begin{align*}
\EE_k(\norm{\ve{e}_{k+1}}{}^2|\mathcal{E}^{k+1}_{I_1}) &= \EE_k(\norm{\ve{e}_k - h_i \ve{a}_i}{}^2|\mathcal{E}^{k+1}_{I_1})\\ 
&= \norm{\ve{e}_k}{}^2 + h_i^2 - 2\EE_k\left(h_i\inner{\ve{e}_k}{\ve{a}_i}|\mathcal{E}^{k+1}_{I_1}\right)\\
&\leq  \norm{\ve{e}_k}{}^2 + \corrQ{q}{\ve{x}_k}^2 + 2\corrQ{q}{\ve{x}_k} \EE_k(\abs{\inner{\ve{e}_k}{\ve{a}_i}}|i \sim \Unif(I_1)).
\end{align*}

 We would like to bound these last two terms.  By the Remark~\ref{corr_med_cor}, for $\alpha \le 1 - q - \beta$,

\rev{$$
 \Pr\left(\corrQ{q}{\ve{x}_k} \le \frac{C_{K} \norm{\ve{e}_k}{}}{\alpha\sqrt{n}}\right) \ge 1 - 2\exp(-m).
$$  As long as $q+\beta$ is bounded away from one, as is the case if the constants in the statement of the theorem are chosen appropriately small, this yields a bound of the form
$$
 \Pr\left(\corrQ{q}{\ve{x}_k} \le \frac{C_{K} \norm{\ve{e}_k}{}}{\sqrt{n}}\right) \ge 1 - 2\exp(-m).
$$}
Also, we apply Lemma \ref{avg_bound} to the set $I_1$ \rev{(recall that $|{I_1}| \ge \beta m$)} to get
that with probability $1-2\exp(-cm)$, 
\[\EE_k\parens{\abs{\inner{\ve{e}_k}{\ve{a}_i}}\,|i \sim \Unif(J)} = \frac{1}{|{I_1}|}\sum_{i\in {I_1}}\abs{\inner{\ve{e}_k}{\ve{a}_i}} \le C \norm{\ve{e}_k}{} \sqrt{\frac{m}{n|{I_1}|}} \le \frac{C \norm{\ve{e}_k}{}}{ \sqrt{\beta n}}.\] 
Thus, \begin{equation}\EE_k(\norm{\ve{e}_{k+1}}{}^2|\mathcal{E}^{k+1}_{I_1}) \le  \parens{1+\frac{\sqrt{\beta}c_2 +  c_3}{\sqrt{\beta}n}}\norm{\ve{e}_k}{}^2.\end{equation}
So, in this case the norm of the error could  increase, but not too much (as we will see below).

So, by the total expectation theorem, we have
\begin{align*}
\EE_k(\norm{\ve{e}_{k+1}}{}^2|\mathcal{E}_{\Accept}(k+1)) &= p_J \EE_k(\norm{\ve{e}_{k+1}}{}^2|\mathcal{E}^{k+1}_{I_1})  + (1-p_J)\EE_k(\norm{\ve{e}_{k+1}}{}^2|\mathcal{E}^{k+1}_{I_2}) \\
&\le \left[p_J \parens{1+\frac{\sqrt{\beta}c_2 +  c_3}{\sqrt{\beta}n}} + (1-p_J)(1- \frac{c_1}{n}) \right] \norm{\ve{e}_k}{}^2 \\
&= \left[ 1 - \frac{c_1}{n} + p_J\left(\frac{(c_1 + c_2) \sqrt{\beta} + c_3}{\sqrt{\beta} n}\right) \right] \norm{\ve{e}_k}{}^2 \\
&\le \left[ 1 - \frac{c_1}{n} + \frac{2\beta}{q}\cdot\frac{c_1 + c_2+ c_3}{\sqrt{\beta} n} \right] \norm{\ve{e}_k}{}^2\\ 
& \le \left[ 1 - \frac{0.5c_1}{n} \right] \norm{\ve{e}_k}{}^2,
\end{align*}
where the last step holds if $\beta$ a sufficiently small constant (we need $\sqrt{\beta} \le c q =: C_q$). Finally, from \eqref{cond_to_accept} we obtain the per-iteration guarantee
\rev{\begin{equation}\label{beta_nonzero}
    \EE_k(\norm{\ve{e}_{k+1}}{}^2) \le
(\lfloor qm) \rfloor/m \left( 1 - \frac{0.5c_1}{n} \right) \norm{\ve{e}_k}{}^2 + (1 - \lfloor qm \rfloor/m) \norm{\ve{e}_{k}}{}^2 \le \left( 1 - \frac{cq}{n} \right) \norm{\ve{e}_k}{}^2.
\end{equation}}
Theorem~\ref{median_convergence} now follows from \eqref{beta_zero} or \eqref{beta_nonzero} by induction.
\end{proof}

\begin{remark}[Condition on $\beta$.] We need the fraction of corruptions $\beta$ to be sufficiently small. Specificlly, our proof of Theorem~\ref{median_convergence} requires $\sqrt{\beta} < c q$, where $c$ is some small positive constant.   Intuitively, this is required since the quantile bound (admissibility) is the only way to bound potential loss if the step is made using one of the corrupted equations (as we do not impose any restrictions on the size of corruptions). Moreover, the expected loss of progress, given the projection on the admissible corrupted equation, must be so small that it is compensated by the expected exponential convergence rate, given that one of the equations from the uncorrupted part was selected. 
\end{remark}

\begin{remark}
Although the bounded density assumption is crucial for Proposition~\ref{sub_conditioning} to hold (see Remark~\ref{bdd_density_assumption}), one should not expect the failure of Proposition \ref{sub_conditioning} to result in the QuantileRK method not converging. In the Bernoulli case, the per-iteration guarantee given likely no longer holds, however one expects it to fail for only a very small set of vectors $\ve{x}_k$. Provided that the QuantileRK method does not attract iterates to this set of bad vectors, one should still expect convergence from a randomly chosen $\ve{x}_0$. We leave such an analysis to future work. (However we empirically demonstrate convergence in Figure~\ref{fig:bernoulli_adversarial} (a).)

\end{remark}



\subsection{Analysis of QuantileSGD method}\label{subsec:QuantileSGD}
In this section, we provide a proof that the QuantileSGD method converges.  To do so, we first introduce an optimal SGD method in Section \ref{subsubsec:OPTSGD} and then prove that QuantileSGD approximates this optimal method in Section \ref{subsubsec:quantileSGD}.  We then give an improved analysis in the streaming setting in Section \ref{sec:streaming}. 
\subsubsection{OptSGD} \label{subsubsec:OPTSGD}
As a first step towards the analysis of the quantile-based SGD method, we introduce the OptSGD method taking the steps of the optimal size towards the solution.
Note that SGD iterates can be written in the form 
\begin{equation}\label{l1sgd}
\ve{x}_{k+1} = \ve{x}_k - \eta_k s_i(\ve{x}_k) \ve{a}_i, \quad\quad \text{
where }\quad  s_i(\ve{x}_k) := \sgn(\inner{\ve{a}_i}{\ve{x}_k}-b_i);
\end{equation}
that is, the vector $s_i(\ve{x}_k)\ve{a}_i$ is directed from the hyperplane defined by the $i^{\text{th}}$ equation towards the half space that $\ve{x}_k$ lies on. We assume that SGD samples rows uniformly, so $i \sim \Unif([m]).$ The constant $\eta_k > 0$ defines the length of the step (recall that $\|\ve{a}_i\|_2 = 1)$). OptSGD chooses the step size $\eta_k^*$ so that the expected distance to the solution $\EE \norm{\ve{e}_{k+1}}{2}^2 = \EE \norm{\ve{x}_{k+1} - \ve{x}^*}{2}^2$ is minimized. 

Namely, we have
\begin{align}\label{expectation_via_optsize}
\EE\left(\norm{\ve{e}_{k+1}}{2}^2\right) &= \EE\left(\norm{\ve{e}_k - s_i(\ve{x}_k)\eta_k\ve{a}_i}{2}^2\right)\nonumber\\
&= \EE\left(\|\ve{e}_k\|^2_2 - 2 s_i(\ve{x}_k) \inner{\ve{e}_k}{\ve{a}_i}\eta_k + s_i(\ve{x}_k)^2\norm{\ve{a}_i}{2}^2\eta_k^2\right)\nonumber\\
&= \norm{\ve{e}_k}{2}^2 - 2 \EE\left(s_i(\ve{x}_k)\inner{\ve{e}_k}{\ve{a}_i}\right)\eta_k + \eta_k^2\nonumber\\ 
&= \left(\eta_k - \EE(s_i(\ve{x}_k)\inner{\ve{e}_k}{\ve{a}_i})\right)^2 - \left(\EE(s_i(\ve{x}_k)\inner{\ve{e}_k}{\ve{a}_i})\right)^2 + \norm{\ve{e}_k}{2}^2,
\end{align} 
which is minimized by setting 
\begin{equation}\label{optSize}
\eta^*(\ve{x}_k) = \EE\left(s_i(\ve{x}_k)\inner{\ve{e}_k}{\ve{a}_i}\right) = \frac1m \sum_{i=1}^m s_i(\ve{x}_k)\inner{\ve{e}_k}{\ve{a}_i}.
\end{equation} 


\subsubsection{Quantile SGD} \label{subsubsec:quantileSGD}
In the previous section, we derived a theoretically optimal step size for $l_1$ stochastic gradient descent.  The formula for the step size \eqref{optSize} relied on $\ve{e}_k$ which is unknown during runtime. Actually, since $\inner{\ve{e}_k}{\ve{a}_i} = \inner{\ve{x}_k}{\ve{a}_i} - \inner{\ve{x}^*}{\ve{a}_i}=\inner{\ve{x}_k}{\ve{a}_i} - b_i$ for any uncorrupted equation, it is the presence of corruptions that makes $\eta^*(\ve{x}_k)$ unavailable at runtime.  Here we show that order statistics can be applied to give an approximation to the optimal step size.


First, let us show that $\eta_k^*(\ve{x}_k)$ is well-approximated by 
$M(\ve{x}_k - \ve{x}^*)$. We notice that the sums defining $\eta_k^*(\ve{x}_k)$ and $M(\ve{x}_k - \ve{x}^*)$ respectively differ only in the terms corresponding to the indices of the corrupted equations. So, given that the fraction of corruptions is small enough, we can efficiently bound this difference.

\begin{proposition}\label{eta_opt_m} Fix any $\delta \in (0,1)$. 
Let the system be defined by random matrix $\ve{A} \in \mathbb{R}^{m \times n}$ 
satisfying Assumptions \ref{incoherent} and \ref{bounded_density} with $m\geq C_{\subgaussconst} n$, and 
$\beta = |\textup{supp}(\ve{b}_C)|/m$ a small enough positive constant.  Let $\eta^*(\ve{x})$ be optimal step size for SGD method defined as in \eqref{optSize}. Then, with probability at least $1 - c\exp(-c_{\subgaussconst} m)$ we have for any $\ve{x} \in \R^n$ that
\begin{equation}\label{M_est}(1 - \delta) \eta^*(\ve{x}) \leq M(\ve{x} - \ve{x}^*) \leq (1 + \delta) \eta^*(\ve{x}).\end{equation}
\end{proposition} 
\begin{proof}
Let $S$ denote the set of indices corresponding to negative terms in the sum \eqref{optSize}. Note that for all uncorrupted equations $i$, we have $s_i(\ve{x}_k) = \sgn(\inner{\ve{e}_k}{\ve{a}_i})$, so the $i$-th term in $\eta^*(\ve{x}_k)$ is non-negative, and $|S| \leq \beta m.$ We then have $$\abs{\eta^*(\ve{x}) - M(\ve{x} - \ve{x}^*)} \leq \frac{2}{m}\sum_{i\in S} \abs{\inner{\ve{x}-\ve{x}^*}{\ve{a}_i}}.$$

Rescaling to normalize $\ve{x}-\ve{x}^*$ and applying Lemma \ref{avg_bound} allows us to further bound 
$$\abs{\eta^*(\ve{x}) - M(\ve{x} - \ve{x}^*)} \leq\frac{2|S|}{m} C_{\subgaussconst} \frac{1}{\sqrt{|S|/m} \sqrt{n}} \norm{\ve{x}-\ve{x}^*}{} \leq 2C_{\subgaussconst}\frac{\sqrt{\beta}}{\sqrt{n}} \norm{\ve{x}-\ve{x}^*}{}$$ uniformly for all $\ve{x}$ with probability at least $1 - 2\exp(-cm)$.

Moreover, by Propostion \ref{M_lower_bound},  \[M(\ve{x} - \ve{x}^*)\ge \frac{c}{\sqrt{n}} \norm{\ve{x}-\ve{x}^*}{}\] for all $\ve{x}$ with probability at least $1-2\exp(-c_{\subgaussconst} m)$. Thus by taking $\beta$ to be a sufficiently small constant (so that the difference between $\eta^*(\ve{x}) $ and $M(\ve{x} - \ve{x}^*)$ is negligible compared to the size of $M(\ve{x} - \ve{x}^*)$), we conclude the proof of Proposition~\ref{eta_opt_m}.
\end{proof}

Although the empirical mean $M(\ve{x} - \ve{x}^*)$, as well as $\eta_k^*$ is not available at runtime, the above proposition allows us to show that in order to obtain a near optimal convergence guarantee, it suffices to approximate $\eta_k^*$ to within a constant factor.

\begin{proposition}\label{almostOptrate} 
Let the system be defined by random matrix $\ve{A} \in \mathbb{R}^{m\times n}$ satisfying Assumptions \ref{incoherent} and \ref{bounded_density} with $m\geq C_k n$.
Suppose we run an SGD method \eqref{l1sgd} with the stepsize $\eta_k$, satisfying $0 < c_1 \leq \eta_k/\eta^*(\ve{x}_k) \leq c_2 < 2$ at each iteration $k = 1, 2, 3, \ldots$, where $\eta^*(\ve{x}_k)$ is an optimal step size given by ~\eqref{optSize}. Then, for any $\beta = |\textup{supp}(\ve{b}_C)|/m \in (0,1)$, there exists a constant $c = c(c_1, c_2) > 0$ such that 
\begin{equation}\label{almost_rate}
\EE(\norm{\ve{e}_{k+1}}{2}^2) \leq \left(1 - c\left(\frac{\eta^*(\ve{x}_k)}{\norm{\ve{e}_k}{2}}\right)^2 \right)\norm{\ve{e}_k}{2}^2.
\end{equation}

Moreover, if the fraction of corrupted equations $\beta$ is small enough, then with probability at least $1-c\exp(-c_{\subgaussconst} m)$, $\ve{A}$ is sampled such that the rate of convergence is linear, namely, there exists a constant $C = C(c_1,c_2) > 0$ such that 
\begin{equation}\label{sqd_almost_optimal_rate}
\EE(\norm{\ve{e}_{k+1}}{2}^2) \leq \left(1 - \frac{C}{n} \right)\norm{\ve{e}_k}{}^2.
\end{equation}
\end{proposition}
\begin{proof}
Throughout the proof, we adopt a shortening notation $\eta^*_k = \eta^*(\ve{x}_k)$. 

Indeed, by the condition on $\eta_k$ we have that 
$$
|\eta_k - \eta_k^*| \le \eta_k^*\max\{c_2 - 1, c_1 - 1\}
$$
and $c = 1 - (\max\{c_2 - 1, c_1 - 1\})^2 > 0$. So, by equation~\eqref{expectation_via_optsize} and the definition of $\eta^*_k$ (in \eqref{optSize}), we have that
$$
\EE\left(\norm{\ve{e}_{k+1}}{2}^2\right) = (\eta_k - \eta_k^*)^2 - (\eta_k^*)^2 + \|\ve{e}_k\|_2^2  \le \|\ve{e}_k\|_2^2 - c(\eta_k^*)^2,
$$
and so
$$
\EE(\norm{\ve{e}_{k+1}}{2}^2) \leq \left(1 - c\left(\frac{\eta_k^*}{\norm{\ve{e}_k}{2}}\right)^2 \right)\norm{\ve{e}_k}{2}^2.
$$

To show that the convergence rate is linear, note that by applying Proposition~\ref{eta_opt_m} with $\delta=1/3$, and Proposition~\ref{M_lower_bound}, we have the bound \[\eta^*(\ve{x}_k) \geq \frac{3}{4}M(\ve{x}_k - \ve{x}^*) \gtrsim \frac{1}{\sqrt{n}}\norm{\ve{x}_k-\ve{x}^*}{}.\] This concludes the proof of Proposition~\ref{almostOptrate}.
\end{proof}



{\bfseries Roadmap.} We are now set to give a proof of Theorem~\ref{SGD_convergence}. The general plan is as follows: we know that quantiles of the residual $\corrQ{q}{\ve{x}_k}$ are well-approximated by the empirical uncorrupted quantiles $\uncorrQ{q}{\ve{x}_k}$ (Lemma~\ref{residual_vs_emp_quantiles}), then we show that empirical uncorrupted quantiles concentrate near the empirical mean $M(\ve{x} - \ve{x}^*)$, which is in turn close enough to the optimal step size $\eta^*(\ve{x})$ (Proposition~\ref{eta_opt_m}). Finally, we invoke Proposition~\ref{almostOptrate} to conclude the linear convergence rate of the QuantileSGD$(q)$ method.

\begin{proof}[Proof of Theorem~\ref{SGD_convergence}]
We upper bound the $q$-quantile of the corrupted residual, 
\begin{equation}\label{byMarkov}
\corrQ{q-\beta}{\ve{x}_k} \le \uncorrQ{q}{\ve{x}_k} \leq \frac{1}{1 - q}M(\ve{x}_k - \ve{x}^*) \leq \frac{1 + \delta}{1 -q}\eta^*(\ve{x}_k) < 2 \eta^*(\ve{x}_k),
\end{equation}
\rev{where the first inequality follows from Lemma~\ref{residual_vs_emp_quantiles}, the second from Markov's inequality, the third from Proposition~\ref{eta_opt_m} with probability at least $1 - 2\exp(-cm)$, and the fourth by choosing $ \delta \in (0, 1 -  2q)$.}

For the lower bound, we have that $\corrQ{q-\beta}{\ve{x}_k} \ge \uncorrQ{q - 2\beta}{\ve{x}_k}$ by Lemma~\ref{residual_vs_emp_quantiles}. Then, $$\uncorrQ{q-2\beta}{\ve{x}} \ge \frac{c_1}{\sqrt{n}}\norm{\ve{x}-\ve{x}^*}{}$$ for all $\ve{x}$ with probability at least  $1-\exp(-cm)$ by Lemma \ref{quantile_lower}. 
Now, $$M(\ve{x}) \le \frac{c_2}{\sqrt{n}}\norm{\ve{x}-\ve{x}^*}{},$$ so we may upper bound $\eta^*(\ve{x})$,
$$\eta^*(\ve{x}) \leq \frac{M(\ve{x})}{1 - \delta} \le \frac{c_2}{(1 - \delta)\sqrt{n}}\norm{\ve{x}-\ve{x}^*}{} \le\frac{c_2}{(1 - \delta)c_1}\uncorrQ{q-2\beta}{\ve{x}_k} \le \frac{1}{c}\corrQ{q-\beta}{\ve{x}_k}$$
for some positive constant $c$. 

Combining these upper and lower bounds on $\corrQ{q-\beta}{\ve{x}}$ we find that there exists $c > 0$ so that for all $\ve{x}\in\R^n,$ \rev{$$0<c<\frac{\corrQ{q-\beta}{\ve{x}}}{\eta^*(\ve{x})}< \frac{1+\delta}{(1-q)\eta^*(\ve{x})} < 2.$$}  We have shown that the hypothesis of Proposition \ref{almostOptrate} holds. Theorem~\ref{SGD_convergence} follows by induction.
\end{proof} 


\begin{remark}\label{gen_quant_for_uniform}
In some cases, for example, when $\ve{a}_i$ are independent vectors sampled uniformly from $S^{n-1}$ we can show that a bigger range of quantiles for an SGD step guarantees exponential convergence of the QuantileSGD method. In particular, using Gaussian concentration instead of Markov's inequality in \eqref{byMarkov},
the statement of Theorem~\ref{SGD_convergence} holds for \emph{QuantileSGD($q - \beta$)} for all $q \in (0, 0.75)$. Note that this justifies the optimal values for the quantile $q$ obtained experimentally (see Figure~\ref{fig:various_quantiles_rk_sgd} b)
\end{remark}

\subsubsection{Streaming Setting} \label{sec:streaming}

In the matrix setting we only prove convergence for a sufficiently small fraction of corruptions $\beta.$  While one could in principle unwind the constants from the random matrix theorems that we have applied, it would be unlikely to result in new insights.  Instead we note that the key complication in the matrix setting was handling ``asymmetries" in the matrix $\ve{A}$.  While the rows were sampled over $S^{n-1}$ in a close-to-uniform way, there was no guarantee that the rows of $\ve{A}$ (representing only a sample from this distribution) were uniformly spread over the sphere. 

Here we present a more optimized analysis in the streaming setting, which may be viewed as a model for extremely tall matrices where each row is likely to be sampled only once in the course of the method. In particular, it allows us to justify the QuantileSGD method when up to a $0.35$ fraction of all equations are corrupted (note that in both Theorem~\ref{median_convergence} and Theorem~\ref{SGD_convergence} we formally asked for the fraction of corruptions $\beta$ to be ``small enough").

Instead of a matrix let us consider some distribution $\mathcal D$ over $\R^n$ and $\beta \in [0,1]$.  On each of many iterations, we receive a pair $(\ve{a}_k, b_k)$ (in a non-streaming setting this pair was a row of the matrix and a corresponding entry of the vector $\ve{b}$ respectively). The vector $\ve{a_k}$ is always sampled from $\mathcal{D}.$  With probability $1-\beta,$ $b$ was selected so that $b = \inner{\ve{a}}{\ve{x}^*},$ and with probability $\beta,$ $b$ was chosen arbitrarily, and possibly adversarially.  Our goal is to  approximate $\ve{x}^*$.

For simplicity, we allow ourselves an arbitrary number of samples to estimate the quantiles of the residual $\corrQ{q}{\ve{x}_k}$, where $\ve{a_k}$ from Definition~\eqref{corrqx} are random gaussian vectors and respective $b_i$ are given by the samples. 



\begin{theorem}
\label{streaming_sgd_convergence}
In the streaming setting with adversarial corruptions and Gaussian samples (namely $\ve{a}_k)$ are standard $n$-variate Gaussian random vectors), QuantileSGD(q) converges to $\ve{x}^*$ with $\beta=0.35$ as long as the quantile $q$ is chosen sufficiently small. 
\end{theorem}

\begin{remark}
The model of the left hand side of the system is different from our earlier convention, in particular, $\ve{a}_k$'s do not have exactly unit norm.  However this distinction is unimportant since (i) our methods are invariant under rescaling the $\ve{a}_k$s and (ii) the one-dimensional projections of the uniform distribution over $\sqrt{n}S^{n-1}$ converge in distribution to a Gaussian as $n\rightarrow\infty$.
\end{remark}


\begin{proof}

Recall that for QuantileSGD, we chose our step size $\eta_k = \corrQ{q}{\ve{x}_k}.$ In the streaming setting with Gaussian samples, the value of $\corrQ{q}{\ve{x}}$ only depends on $\norm{\ve{x}-\ve{x}^*}{}.$ This follows directly from the definition of $\tilde Q_q$ along with rotation-invariance of Gaussian vectors.  Furthermore, $\corrQone{q}$ respects dilations about $\ve{x}^*$ in the sense that \[\corrQ{q}{\ve{x}^* + \lambda (\ve{x}-\ve{x}^*)} = \lambda\corrQ{q}{\ve{x}}\] for $\lambda\in \R.$  Again this is a simple check from the definition of $\tilde Q.$  The same properties hold for the optimal SGD step size $\eta^*$ as per~\eqref{optSize} for the same reasons.

These properties imply that $\corrQ{q}{\ve{x}}/\eta^*(\ve{x})$ is constant over $\R^n \setminus \{\ve{x}^*\}.$  We are going to show that for small $q$ this quantity lies strictly between $0$ and $2.$ In other words \begin{equation}\label{convergence_check}0 < c < \corrQ{q}{\ve{x}_k}/\eta^*(\ve{x}_k) < C < 2\end{equation} for all iterates $\ve{x}_k.$ (Of course this bound holds for all $\ve{x}\in\R^n$, but we emphasize that we apply this bound to the iterates.) This will allow us to apply  Proposition \ref{almostOptrate} (in the form of \eqref{almost_rate}) to conclude that QuantileSGD(q) converges for $q$ small enough.

The lower bound of \eqref{convergence_check} clearly holds for $q$ positive, since $\corrQ{q}{\ve{x}_k}/\eta^*(\ve{x}_k)$ is nonzero as long as $\ve{x}_k\neq \ve{x}^*$ (and of course $\eta^*(\ve{x}_k) < \infty$). 

Also recall that for all uncorrupted equations we have $\eta^*(\ve{x}_k) =\EE|\inner{\ve{x}_k - \ve{x}^*}{\ve{a}_k}|$. So, we can lower bound
\begin{align*}
\eta^*(\ve{x}_k) &\ge (1-\beta)\EE(\abs{\inner{\ve{e}_k}{\ve{a}_k}}) + \beta\EE(-\abs{\inner{\ve{e}_k}{\ve{a}_k}}) \\
&= (1-2\beta)\EE(\abs{\inner{\ve{e}_k}{\ve{a}_k}}) \\
&= (1-2\beta)\sqrt{\frac{2}{\pi}} \norm{\ve{e}_k}{},
\end{align*} where the last constant is the expectation of a standard half-normal random variable. 

\rev{By Lemma \ref{residual_vs_emp_quantiles}}, we also have \[\corrQ{q}{\ve{x}_k} \leq \uncorrQ{q+\beta}{\ve{x}_k} = \norm{\ve{e}_k}{}\Phi_{q+\beta},\] where $\Phi_q$ denotes the $q$-quantile of the standard half-normal distribution, $\abs{\mathcal{N}(0,1)}$.
 The upper bound in equation \eqref{convergence_check} is equivalent to the inequality \begin{equation}\label{streaming_convergence_requirement}\norm{\ve{e}_k}{} \Phi_{q+\beta} < C(1-2\beta)\sqrt{\frac{2}{\pi}} \norm{\ve{e}_k}{}, \end{equation} where $C$ is allowed to be any constant smaller than $2$ (e.g $1.99$).  This inequality is true for small positive $q$ as long as \[\uncorrQ{\beta}{\abs{\mathcal{N}(0,1)}} < \sqrt{\frac{8}{\pi}}(1-2\beta).\] One can verify numerically that the inequality holds for $\beta=0.35,$ and indeed for slightly larger values. This concludes the proof of Theorem~\ref{streaming_sgd_convergence}.
\end{proof}

\begin{remark}
One can find explicit pairs $q,\beta$ that work by solving the inequality \eqref{streaming_convergence_requirement} numerically.  For instance quantiles $0.1, 0.3,$ and $0.5$ can handle corruption rates of roughly $0.32, 0.25,$ and $0.18$ respectively.
\end{remark}

\begin{remark}
An adversary generating corruptions at runtime can make the bounds in the proof of \ref{streaming_sgd_convergence} as tight as desired.  Thus one cannot expect convergence in general if $\beta$ is much larger than $0.35.$
\end{remark}

\begin{remark}
While the above analysis gives results that are on the same order of magnitude as experiments show, this setting is far more adversarial than what one would encounter in practice.  Our experiments demonstrate that one can tolerate higher levels of corruptions than what our theory predicts in this setting.  Extending the analysis to the setting of our experiments would require fixing a particular model for the corruptions.  By considering adversarial corruptions generated at run-time, we handle any such model.
\end{remark}





\section{Implementation Considerations}\label{sec:imp}

In this section, we discuss several important considerations regarding the implementation of QuantileRK and QuantileSGD.  In particular, we touch on the streaming setting in which the rows of the measurement matrix are sampled from a distribution and provided in an online manner.  We additionally discuss various considerations for constructing the sample of the residual, and the choice of quantile to apply in each method.  

\subsection{Streaming setting} First, we note that the streaming setting described in Section~\ref{sec:streaming} provides a good model for many of our experiments.  For example, in several of the experiments below, we sample 2000 rows (2000 iterations) from a 50000 row Gaussian matrix.  We expect most rows to be sampled only once, which places us within the context of the streaming setting.  For this reason, we expect that our methods can in practice handle a larger fraction of corruptions than is reflected in Theorems~\ref{median_convergence} and \ref{SGD_convergence}. 

\subsection{Sample size} Next, we mention several approaches for decreasing the computational burden of computing the residual in each iteration of QuantileRK and QuantileSGD.  Note that both QuantileRK and QuantileSGD as written in Methods~\ref{QuantileRK} and \ref{QuantileSGD} use a sample of the residual of size $t$.  This is much more efficient than constructing the entire residual in each iteration, with the cost scaling with $tn$ instead of $mn$ when constructing the entire residual.  

The optimal sample size depends upon the quantile chosen, the fraction of corruptions, and the number of iterations employed.  Given the fraction of corruptions, one should choose the sample size and quantile so that the number of corruptions in the sample is at most $(1-q)t$ with high probability (this could be calculated with a Chernoff bound).  In particular, more aggressive methods with higher choice of quantile demand larger sample size to ensure that corruptions may be avoided with the quantile calculation.  

\subsection{Quantile selection} For QuantileRK, a larger quantile corresponds to a more aggressive method which is more likely to make the sampled projection.  The quantile can be chosen quite close to one if very few corruptions are expected.  Meanwhile, for QuantileSGD, the OptSGD theory demonstrates that the optimal quantile to select is the mean of the uncorrupted residual. In the case of Gaussian rows with no corruptions, the mean happens to coincide with the $0.58$ quantile.  So for QuantileSGD the quantile should be chosen near $0.5$ if few corruptions are expected.


\subsection{Sliding window} Now, as mentioned previously, constructing the sample of the residual requires $\mathcal{O}(tn)$ computation.  We can decrease this per-iteration cost by reusing residual entries between iterations.  This suggests using a `sliding window' approach where the sample from which we compute the quantile consists of residual entries collected over multiple iterations.  We implement this approach in the experiments below, using on the order of several hundred of the most recently computed residuals.  One might expect that this causes significant loss in performance due to the varying scale of the residuals in each iteration, but empirically we see nearly identical performance for moderately sized windows (on the order of 100-500 iterations).  

The sliding window approach raises the question of what to do in the initial iterations before the iteration number has reached the window size.  One could populate the entire window in the first iteration by sampling as many residual entries as the window size, and then just replacing residual entries as new ones are sampled in the next iterations.  Alternatively, one could simply use a partial window until the iteration number reaches the window size.  However, this could significantly slow convergence if there are corruptions that are large relative to the initial error $\|\ve{x}_0 - \ve{x}^*\|$ that get sampled in these initial iterations.

\section{Experimental Results}\label{sec:experiment}

Each experiment is run using using Python version 3.6.9 on a single 24-core machine.  

\subsection{Comparing various quantiles}

Our theoretical analysis does not provide specific guidance for choice of quantiles (besides rough relationships between $q$ and $\beta$), so we investigate the problem of choosing quantiles empirically. Figure~\ref{fig:various_quantiles_rk_sgd} shows the behaviors of QuantileRK and QuantileSGD for various corruption rates $\beta$ and choices of quantile $q.$  For each $\beta,$ we plot the log relative error after $2000$ iterations as a function of $q$ (this is the quantity $\log(\norm{\ve{x}_{2000}-\ve{x}^*}{}/\norm{\ve{x}_0-\ve{x}^*}{})$. In order to de-noise the plots, each plotted point is the median over $10$ trials. On each trial we generate a new $50000\times 100$ Gaussian system with a $\beta$ fraction of corrupted entries.  A corrupted entry of $\ve{b}$ is modified by adding a uniformly random value in $[-5,5].$

It is interesting to point out the optimal quantiles for various corruption rates.   In the case of QuantileRK, we see that the optimal quantile tends to be just shy of $1-\beta.$  This aligns with the intuition that QuantileRK should be as aggressive as possible while avoiding projections onto badly corrupted hyperplanes.  It is clear that QuantileRK cannot choose a quantile larger than $\beta$, otherwise we are likely to sample in the $\beta$ fraction of corrupted rows, resulting in a threshold which is too large.  In practice it is often best to choose a quantile which is somewhat smaller that what the graph suggests.  As the quantile approaches $1-\beta$ the risk of performing a bad projection becomes large enough that we observe bad projections within a few thousand iterations.

We see that QuantileSGD is much more robust to the choice of quantile.  For instance when $\beta = 0.1,$ the optimal quantile appears to be near $0.5.$  However we see near-optimal convergence behavior as long as $\beta$ is between $0.3$ and $0.7.$
%

%

\begin{figure}
		\centering
	\subfloat[QuantileRK]{\includegraphics[width=0.45\columnwidth]{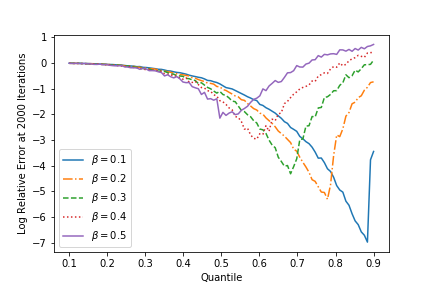}}	\qquad
	\subfloat[QuantileSGD]{\includegraphics[width=0.45\columnwidth]{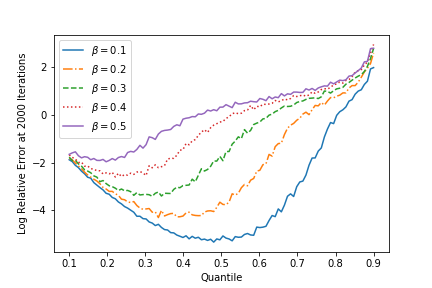}}
	\caption{$\log(\norm{\ve{x}_{2000}-\ve{x}^*}{}/\norm{\ve{x}_0 - \ve{x}^*}{})$ for (a) QuantileRK and (b) QuantileSGD run on $50000\times 100$ Gaussian system, with various corruption rates $\beta$ and quantile choices.}
	\label{fig:various_quantiles_rk_sgd}
\end{figure}

%

\subsection{General convergence plots}
\begin{figure}
		\centering
	\subfloat[Gaussian model]{\includegraphics[width=0.45\columnwidth]{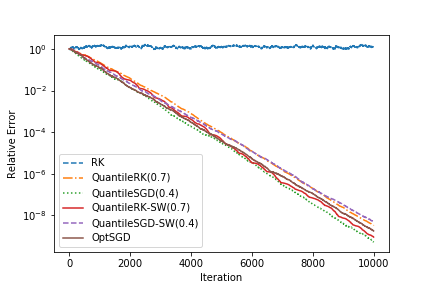}}	\qquad
	\subfloat[Coherent model]{\includegraphics[width=0.45\columnwidth]{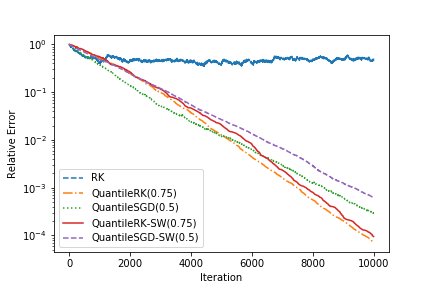}}
	\caption{Relative error as a function of iteration count plotted for a $50000\times 100$ Gaussian and coherent model with a $0.2$ corruption rate.  The coherent system was generated by sampling entries uniformly in $[0,1)$ and then normalizing the rows of the resulting matrix.}
	\label{fig:gaussian_coherent}
\end{figure}

\begin{figure}
		\centering
	\subfloat[Bernoulli model]{\includegraphics[width=0.45\columnwidth]{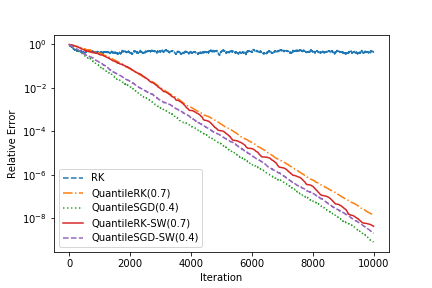}}	\qquad
	\subfloat[Adversarial model]{\includegraphics[width=0.45\columnwidth]{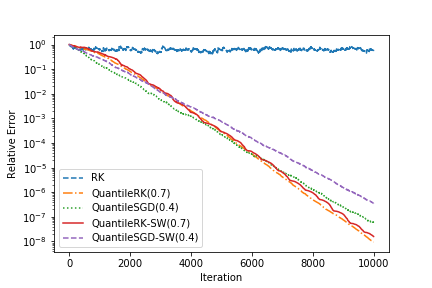}}
	\caption{Relative error as a function of iteration count plotted for a $50000\times 100$ Bernoulli and adversarial model with a $0.2$ corruption rate. Each entry of the Bernoulli matrix is generated to be $-1$ or $1$ before normalizing rows. For the adversarial model, a random subset of rows from the corresponding Gaussian system were selected and corrupted to yield a $0.2m$ sized consistent subsystem.}
    \label{fig:bernoulli_adversarial}
\end{figure}

In Figure \ref{fig:gaussian_coherent} and Figure \ref{fig:bernoulli_adversarial} we show the convergence behavior of our methods on a $50000 \times 100$ system with a $\beta=0.2$ fraction of corruptions. In Figure \ref{fig:gaussian_coherent} and Figure \ref{fig:bernoulli_adversarial} (a) entries are corrupted by adding a uniformly random value in $[-5,5].$

The label ``RK" signifies the standard Randomized Kaczmarz method without thresholding.  The methods marked QuantileRK-SW and QuantileSGD-SW are the ``sliding window" versions of QuantileRK and QuantileSGD.  The methods marked QuantileRK and QuantileSGD are the sampled variants.  We set our window size and sample size to $400$ for these experiments.  Finally, we include OptSGD only in Figure \ref{fig:gaussian_coherent} (a).

In Figure \ref{fig:gaussian_coherent} (a) we show a normalized Gaussian system (i.e., a system with rows sampled uniformly over $S^{n-1}$). We observe that all four of our quantile methods exhibit similar convergence behavior.  Notably, these methods perform comparably to OptSGD, which chooses an optimal step size at each iteration.  (Of course OptSGD cannot be run in practical settings, as it requires knowledge of $\ve{x}^*$.)

In Figure \ref{fig:gaussian_coherent} (b) we consider a system with ``coherent rows".  This matrix is created by generating each entry i.i.d.\ uniformly in $[0,1],$ and then normalizing the rows of the resulting matrix.  We call the system coherent because pairs of rows typically have large inner product with one another.  Such a matrix does not have isotropic rows, and is therefore not covered by our theoretical analysis.  Nonetheless, we do observe convergence, albeit at a slower rate than for the Gaussian model. 

In Figure \ref{fig:bernoulli_adversarial} (a) we show a Bernoulli system.  Here each entry of our matrix is sampled uniformly in $\{-1,1\}$ and the rows are normalized.  This matrix violates the ``bounded density" assumption of our theoretical analysis.  However we still see convergence behavior which is comparable to the Gaussian case.

Figure \ref{fig:bernoulli_adversarial} (b) shows a Gaussian system which is corrupted adversarially.  In this model, we choose a random collection of indices to corrupt.  The corruptions are then chosen so that the corrupted subsystem is consistent.  This model is adversarial in the sense that it tries to ``trick" the method into believing that there is another solution in addition to $\ve{x}^*$. (Note that this is different from the truly adversarial corruptions discussed in the context of the streaming setting.)  Our theory does address this case, and here we see convergence is comparable to the randomly corrupted Gaussian case.

%

\begin{figure}
		\centering
	\subfloat[Effect of aspect ratio]{\includegraphics[width=0.45\columnwidth]{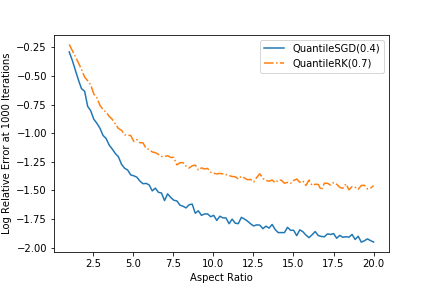}}	\qquad
	\subfloat[Effect of corruption size]{\includegraphics[width=0.45\columnwidth]{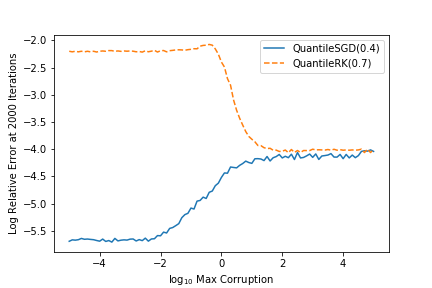}}
	\caption{(a) Log relative error for QuantileSGD and QuantileRK after $1000$ iterations on a $100a\times 100$ Gaussian system with a $0.2$ corruption rate, where $a=m/n$ is the aspect ratio of the matrix. (b) Log relative error for QuantileSGD and QuantileRK after $2000$ iterations, as a function of corruption size.  We use a $50000\times 100$ Gaussian system and corrupt our system by adding a uniform value in $[-10^x, 10^x]$. }
	\label{fig:aspect_error_sizes}
\end{figure}

\subsection{Influence of the aspect ratio}
Each of our experiments so far dealt with extremely tall $50000\times 100$ matrices.  Since we ran at most $10000$ iterations we were unlikely to sample a given row many times.  Thus our experiments have effectively been run in the streaming setting.  A strength of our theory was providing convergence guarantees even for matrices which are not too tall. In Figure \ref{fig:aspect_error_sizes} (a) we show the convergence behavior of QuantileSGD and QuantileRK as a function of the aspect ratio.  In this plot we consider random Gaussian matrices with a $\beta=0.2$ fraction of corruptions which are $100a\times 100,$ where $a$ is the aspect ratio. Each data point is the median error taken over $100$ separate trials.

%

\subsection{Effect of corruption size}
In Figure \ref{fig:aspect_error_sizes} (b) we illustrate the behaviors of QuantileSGD and QuantileRK as the corruption sizes are varied.  For each value on the $x$-axis, $x$, we corrupt the vector $\ve{b}$ by adding values sampled uniformly from $[-10^{x}, 10^{x}]$ to a $\beta=0.2$ fraction of entries.  As we see, both of our methods still converge well even when the corruption sizes are very large.  Their behavior for very small errors is perhaps surprising.

In particular, QuantileRK seems to perform better when the corruptions are very large.\footnote{This type of behavior was noted in \cite{haddock2018randomized}, although for different reasons.}  The reason for this is that when the corruptions are very small relative to $\norm{\ve{x}_k - \ve{x}^*}{},$ the system behaves as though it is consistent.  For a consistent system QuantileRK behaves too conservatively by rejecting $30$ percent of the rows.  When the size of corruptions becomes comparable to or larger than $\norm{\ve{x}_0 - \ve{x}^*}{},$ this behavior disappears.

QuantileSGD on the other hand behaves better for consistent systems as rows are never rejected.  The more consistent the system, the more likely a given step is to move the iterate closer to $\ve{x}^*.$  We see this behavior for QuantileRK and QuantileSGD in Figure \ref{fig:various_quantiles_rk_sgd} as well.

\subsection{Real world data}

\begin{figure}
		\centering
	\subfloat[Tomography system]{\includegraphics[width=0.45\columnwidth]{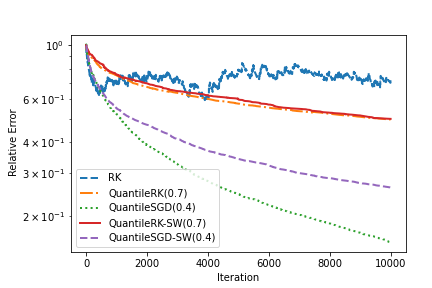}}	\qquad
	\subfloat[Wisconsin Breast Cancer dataset]{\includegraphics[width=0.45\columnwidth]{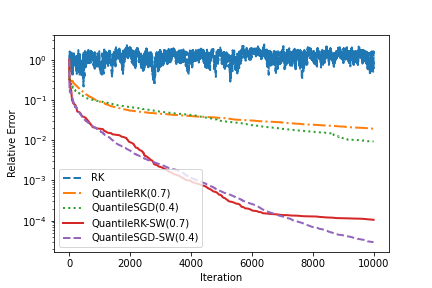}}
	\caption{(a) Relative error for each method run on a $1200\times 400$ system designed for tomography.  Corruptions were added to $100$ uniformly random entries of $\ve{b}.$ (b) Relative error for each method run on a $699\times 10$ matrix obtained from the Wisconsin Breast Cancer dataset. Corruptions were added to $100$ uniformly random entries of $\ve{b}.$}
	\label{fig:tomography_wisconsin}
\end{figure}

%

%
%

Finally, in Figure \ref{fig:tomography_wisconsin} we illustrate our methods on two real world data sets.  In Figure \ref{fig:tomography_wisconsin} (a), we experiment on a tomography problem generated using the Matlab
Regularization Toolbox by P.C. Hansen (\url{http://www.imm.dtu.dk/~pcha/Regutools/})
\cite{hansen2007regularization}. We present a 2D tomography problem $\ve{A}\ve{x} = \ve{b}$ for an $m \times n$ matrix with $m = f N^2$ and $n = N^2$. Here $\ve{A}$ corresponds to the absorption along a random line through an $N \times N$ grid. In this experiment, we set $N = 20$ and the oversampling factor $f = 3$, which yields a matrix $\ve{A} \in \mathbb{R}^{1200 \times 400}$.  As the resulting system was consistent, we randomly sampled $100$ indices uniformly from among the rows of $\ve{A}$ and corrupted the right-hand side vector $\ve{b}$ in these entries by adding a uniformly random value in $[-5,5]$.  

In Figure \ref{fig:tomography_wisconsin} (b) we use a corrupted system generated from the Wisconsin (Diagnostic) Breast Cancer data set, which includes data points whose features are computed from a digitized image of a fine needle aspirate (FNA) of a breast mass and describe characteristics of the cell nuclei present in the
image \cite{UCI}.  This collection of data points forms our matrix $\ve{A} \in \mathbb{R}^{699 \times 10}$, we construct $\ve{b}$ to form a consistent system, and then corrupt a random selection of $100$ entries of the right-hand side by adding a uniformly random value in $[-5,5]$.  

The label ``RK" signifies the standard Randomized Kaczmarz method without thresholding.  The methods marked QuantileRK-SW and QuantileSGD-SW are the ``sliding window" versions of QuantileRK and QuantileSGD.  The methods marked QuantileRK and QuantileSGD are the sampled variants.  We set our window size and sample size to $100$ for these experiments.  
Again, all four of our proposed methods converge; however, the difference in empirical convergence rate is clearly discernible on this data.  \rev{It is notable that in Figure \ref{fig:tomography_wisconsin} (b) the sliding window variants of the method converge more quickly. Since the sliding window quantile estimate lags into the past where residual entries had larger magnitude, it will typically yield a larger quantile than resampling on each iteration. The effect is to allow for more aggressive projections and step sizes in Quantile-RK and Quantile-SGD.}

\section{Conclusion}\label{sec:conclusion}

In this work, we propose two new methods, QuantileRK and QuantileSGD, for solving large-scale systems of equations which are inconsistent due to sparse, arbitrarily large corruptions in the measurement vector.  Such corrupted systems of equations arise in practice in many applications, but are especially abundant and challenging in areas such as distributed computing, internet of things, and other network problems facing potentially adversarial corruption.

The QuantileRK and QuantileSGD methods make use of a quantile statistic of a sample of the residual in each iteration.  We prove that each method enjoys exponential convergence with mild assumptions on the distribution of the entries of the measurement matrix $\ve{A}$, the quantile parameter of the method $q$, and the fraction of corruptions $\beta$.  

Our experiments support these theoretical results, as well as illustrate that the methods converge in many scenarios not captured by our theoretically required assumptions.  In particular, these methods are able to handle fractions of corruption larger than those predicted theoretically, and converge for systems defined by structured and real measurement matrices which are far from the random matrices for which our theoretical results hold.  We note that both theoretically and experimentally we see that the magnitude of the corruptions do not negatively impact convergence.


\rev{While our experiments show that QuantileRK and QuantileSGD yield good results on many types of corrupted systems, our theory is currently limited to near-Gaussian random matrix model.  One could hope to extend the theory in several directions.}

\rev{First, it would be nice to show a convergence result in the Bernoulli random model. The main obstacle is that we can no longer hope for a per-iteration guarantee that holds over all potential iterates. One would need to show that the set of points at which the per-iteration guarantee fails is small, and that the dynamics of our algorithms are unlikely to be biased towards these ``bad" points.}

\rev{Second, one could also hope to give a non-random characterization of matrices for which our algorithms have good convergence properties.  To handle adversarial corruptions it is probably necessary to assume some type of incoherence. Otherwise the corruptions could be structured to align in a particular direction which points away from $x^*.$  Alternatively, is it possible to detect coherent row subsets of $A$ in order to preempt the effect of structured corruptions?}

\rev{Finally, one might also consider a non-random model for $A$, but where the corrupted entries of $b$ are non-random.  In this setting it seems reasonable that the theory should continue to hold for structured $A$.
One could also attempt to generalize our results to systems of inequalities, and to partially-greedy row sampling schemes. 
}




\bibliographystyle{amsalpha}
\bibliography{revised_main}

\newcommand{\etalchar}[1]{$^{#1}$}
\providecommand{\bysame}{\leavevmode\hbox to3em{\hrulefill}\thinspace}
\providecommand{\MR}{\relax\ifhmode\unskip\space\fi MR }
\providecommand{\MRhref}[2]{%
  \href{http://www.ams.org/mathscinet-getitem?mr=#1}{#2}
}
\providecommand{\href}[2]{#2}
\begin{thebibliography}{DGBSX12}

\bibitem[ABH05]{amaldi}
Edoardo Amaldi, Pietro Belotti, and Raphael Hauser, \emph{Randomized relaxation
  methods for the maximum feasible subsystem problem}, Integer programming and
  combinatorial optimization, Lecture Notes in Comput. Sci., vol. 3509,
  Springer, Berlin, 2005, pp.~249--264. \MR{2210026}

\bibitem[Agm54]{Agmon1954}
S~Agmon, \emph{The relaxation method for linear inequalities}, Canadian J.
  Math. (1954), 382--392.

\bibitem[AK95]{amaldikann}
Edoardo Amaldi and Viggo Kann, \emph{The complexity and approximability of
  finding maximum feasible subsystems of linear relations}, Theor. Comput. Sci.
  \textbf{147} (1995), no.~1-2, 181--210.

\bibitem[BCN18]{bottou2018optimization}
L{\'e}on Bottou, Frank~E Curtis, and Jorge Nocedal, \emph{Optimization methods
  for large-scale machine learning}, SIAM Rev. \textbf{60} (2018), no.~2,
  223--311.

\bibitem[Bot10]{bottou2010large}
L{\'e}on Bottou, \emph{Large-scale machine learning with stochastic gradient
  descent}, Proc. of COMPSTAT, Springer, 2010, pp.~177--186.

\bibitem[BR73]{barrodale1973improved}
Ian Barrodale and Frank~DK Roberts, \emph{An improved algorithm for discrete
  $\ell_1$ linear approximation}, SIAM J. Numer. Anal. \textbf{10} (1973),
  no.~5, 839--848.

\bibitem[BS80]{bloomfield1980least}
Peter Bloomfield and William Steiger, \emph{Least absolute deviations
  curve-fitting}, SIAM J. Sci. Stat. Comp. \textbf{1} (1980), no.~2, 290--301.

\bibitem[BV04]{boyd2004convex}
Stephen Boyd and Lieven Vandenberghe, \emph{Convex optimization}, Cambridge
  university press, 2004.

\bibitem[BW18a]{bai2018greedy}
Zhong-Zhi Bai and Wen-Ting Wu, \emph{On greedy randomized {K}aczmarz method for
  solving large sparse linear systems}, SIAM J. Sci. Comput. \textbf{40}
  (2018), no.~1, A592--A606.

\bibitem[BW18b]{bai2018relaxed}
\bysame, \emph{On relaxed greedy randomized {K}aczmarz methods for solving
  large sparse linear systems}, Appl. Math. Lett. \textbf{83} (2018), 21--26.

\bibitem[CEG83]{Censor1983}
Yair Censor, Paul P~B Eggermont, and Dan Gordon, \emph{Strong underrelaxation
  in {Kaczmarz's} method for inconsistent systems}, Numer. Math. \textbf{41}
  (1983), 83--92.

\bibitem[Cen81]{Censor1981}
Y~Censor, \emph{Row-action methods for huge and sparse systems and their
  applications}, SIAM Rev. \textbf{23} (1981), 444--466.

\bibitem[CLZL19]{chi2019median}
Yuejie Chi, Yuanxin Li, Huishuai Zhang, and Yingbin Liang,
  \emph{Median-truncated gradient descent: A robust and scalable nonconvex
  approach for signal estimation}, Appl. Numer. Harmon. An., Springer, 2019,
  pp.~237--261.

\bibitem[CP12]{RefWorks:498}
X~Chen and A~Powell, \emph{Almost sure convergence of the {Kaczmarz} algorithm
  with random measurements}, J.Fourier Anal. Appl. (2012), 1--20.

\bibitem[CRTV05]{candes2005error}
Emmanuel Candes, Mark Rudelson, Terence Tao, and Roman Vershynin, \emph{Error
  correction via linear programming}, FOCS, IEEE, 2005, pp.~668--681.

\bibitem[CT05]{candes2005decoding}
Emmanuel Candes and Terence Tao, \emph{Decoding by linear programming}, arXiv
  preprint math/0502327 (2005).

\bibitem[DGBSX12]{dekel2012optimal}
Ofer Dekel, Ran Gilad-Bachrach, Ohad Shamir, and Lin Xiao, \emph{Optimal
  distributed online prediction using mini-batches}, J. Mach. Learn. Res.
  \textbf{13} (2012), 165--202.

\bibitem[DLHN17]{DeLoera}
J.~A. De~Loera, J.~Haddock, and D.~Needell, \emph{A sampling
  {K}aczmarz-{M}otzkin algorithm for linear feasibility}, SIAM J. Sci. Comput.
  \textbf{39} (2017), no.~5, S66--S87.

\bibitem[DSS20]{du2020pseudoinverse}
Kui Du, Wutao Si, and Xiaohui Sun, \emph{Pseudoinverse-free randomized extended
  block {K}aczmarz for solving least squares}, arXiv preprint arXiv:2001.04179
  (2020).

\bibitem[EHL81]{eggermont1981iterative}
Paulus Petrus~Bernardus Eggermont, Gabor~T Herman, and Arnold Lent,
  \emph{Iterative algorithms for large partitioned linear systems, with
  applications to image reconstruction}, Linear Algebra Appl. \textbf{40}
  (1981), 37--67.

\bibitem[EK12]{eldar2012compressed}
Yonina~C Eldar and Gitta Kutyniok, \emph{Compressed sensing: theory and
  applications}, Cambridge University Press, 2012.

\bibitem[Elf80]{elf}
Tommy Elfving, \emph{Block-iterative methods for consistent and inconsistent
  linear equations}, Numer. Math. \textbf{35} (1980), no.~1, 1--12.

\bibitem[FCM{\etalchar{+}}92]{Feinco}
H~G Feichtinger, C~Cenker, M~Mayer, H~Steier, and Thomas Strohmer, \emph{New
  variants of the {POCS} method using affine subspaces of finite codimension
  with applications to irregular sampling}, P. Soc. Photo-Opt. Ins., vol. 1818,
  International Society for Optics and Photonics, 1992, pp.~299--311.

\bibitem[FR13]{foucart2013math}
S.~Foucart and H.~Rauhut, \emph{A mathematical introduction to compressive
  sensing}, Springer, 2013, In press.

\bibitem[FS95]{feichtinger1995kaczmarz}
Hans~G Feichtinger and Thomas Strohmer, \emph{A {Kaczmarz}-based approach to
  nonperiodic sampling on unions of rectangular lattices}, SampTA?95: 1995
  Workshop on Sampling Theory and Applications, 1995, pp.~32--37.

\bibitem[GBH70a]{GBH70:Algebraic-Reconstruction}
R.~Gordon, R.~Bender, and G.~T. Herman, \emph{Algebraic reconstruction
  techniques ({ART}) for three-dimensional electron microscopy and {X}-ray
  photography}, J. Theoret. Biol. \textbf{29} (1970), 471--481.

\bibitem[GBH70b]{Gordon1970}
Richard Gordon, Robert Bender, and Gabor~T. Herman, \emph{{Algebraic
  Reconstruction Techniques (ART) for three-dimensional electron microscopy and
  X-ray photography}}, J. Theoret. Biol \textbf{29} (1970).

\bibitem[GSN88]{gentle1988algorithms}
James~E Gentle, VA~Sposito, and Subhash~C Narula, \emph{Algorithms for
  unconstrained $l_1$ simple linear regression}, J. Comp. Stat. Data Anal.
  \textbf{6} (1988), no.~4, 335--339.

\bibitem[Han07]{hansen2007regularization}
Per~Christian Hansen, \emph{Regularization tools version 4.0 for {M}atlab 7.3},
  Numer. Algorithms \textbf{46} (2007), no.~2, 189--194.

\bibitem[HM93]{herman1993algebraic}
Gabor~T Herman and Lorraine~B Meyer, \emph{Algebraic reconstruction techniques
  can be made computationally efficient (positron emission tomography
  application)}, IEEE T. Med. Imaging \textbf{12} (1993), no.~3, 600--609.

\bibitem[HM19]{haddock2019greed}
Jamie Haddock and Anna Ma, \emph{Greed works: An improved analysis of
  {S}ampling {Kaczmarz}-{Motzkin}}, arXiv preprint arXiv:1912.03544 (2019).

\bibitem[HN90]{Hanke1990}
Martin Hanke and Wilhelm Niethammer, \emph{On the acceleration of {Kaczmarz's}
  method for inconsistent linear systems}, Linear Algebra Appl. \textbf{130}
  (1990), 83--98.

\bibitem[HN18a]{HN18Corrupted}
J.~Haddock and D.~Needell, \emph{Randomized projection methods for linear
  systems with arbitrarily large sparse corruptions}, SIAM J. Sci. Comput.
  \textbf{41} (2018), no.~5, S19--S36.

\bibitem[HN18b]{haddock2018randomized}
Jamie Haddock and Deanna Needell, \emph{Randomized projections for corrupted
  linear systems}, AIP Conf. Proc., 1978, no.~1, AIP Publishing, 2018,
  p.~470071.

\bibitem[HN19]{HN18Motzkin}
J.~Haddock and D.~Needell, \emph{On {M}otzkin's method for inconsistent linear
  systems}, BIT \textbf{59} (2019), no.~2, 387--401.

\bibitem[HNRS20]{HNRS20}
J.~Haddock, D.~Needell, E.~Rebrova, and W.~Swartworth, \emph{Stochastic
  gradient descent methods for corrupted systems of linear equations}, Proc.
  Conf. on Information Sciences and Systems, 2020.

\bibitem[JCC15]{cloninger}
Noreen Jamil, Xuemei Chen, and Alexander Cloninger, \emph{Hildreth's algorithm
  with applications to soft constraints for user interface layout}, J. Comput.
  Appl. Math. \textbf{288} (2015), 193--202.

\bibitem[{K}ac37]{kaczmarzoriginal}
S.~{K}aczmarz, \emph{Angen\"aherte aufl\"osung von systemen linearer
  gleichungen}, Bull. Internat. Acad. Polon. Sci. Lettres A (1937), 335--357.

\bibitem[KL20]{kawaguchi2020ordered}
Kenji Kawaguchi and Haihao Lu, \emph{Ordered {SGD}: A new stochastic
  optimization framework for empirical risk minimization}, Int. Conf. on
  Artificial Intelligence and Statistics, 2020, pp.~669--679.

\bibitem[KS18]{krvzic2018l1}
Ana~Sovi{\'c} Kr{\v{z}}i{\'c} and Damir Ser{\v{s}}i{\'c}, \emph{L1 minimization
  using recursive reduction of dimensionality}, Signal Process. \textbf{151}
  (2018), 119--129.

\bibitem[LA04]{li2004maximum}
Yinbo Li and Gonzalo~R Arce, \emph{A maximum likelihood approach to least
  absolute deviation regression}, Eurasip. J. Adv. Sig. Pr. \textbf{2004}
  (2004), no.~12, 948982.

\bibitem[LCZL20]{li2020non}
Yuanxin Li, Yuejie Chi, Huishuai Zhang, and Yingbin Liang, \emph{Non-convex
  low-rank matrix recovery with arbitrary outliers via median-truncated
  gradient descent}, Information and Inference: A Journal of the IMA \textbf{9}
  (2020), no.~2, 289--325.

\bibitem[Lic13]{UCI}
M~Lichman, \emph{{\{UCI\} Machine Learning Repository}}, 2013.

\bibitem[LR19]{loizou2019revisiting}
Nicolas Loizou and Peter Richt{\'a}rik, \emph{Revisiting randomized gossip
  algorithms: General framework, convergence rates and novel block and
  accelerated protocols}, arXiv preprint arXiv:1905.08645 (2019).

\bibitem[MI{\etalchar{+}}20]{morshed2020generalization}
Md~Sarowar Morshed, Md~Saiful Islam, et~al., \emph{On generalization and
  acceleration of randomized projection methods for linear feasibility
  problems}, arXiv preprint arXiv:2002.07321 (2020).

\bibitem[MINEA19]{morshed2019accelerated}
M.~S. Morshed, M.~S. Islam, and M.~Noor-E-Alam, \emph{Accelerated sampling
  {K}aczmarz {M}otzkin algorithm for the linear feasibility problem}, J. Global
  Optim. (2019), 1--22.

\bibitem[MS54]{Motzkin1953}
Theodore~S Motzkin and Isaac~J Schoenberg, \emph{The relaxation method for
  linear inequalities}, Canadian J. Math. \textbf{6} (1954), 393--404.

\bibitem[Nat86]{Nat}
Frank Natterer, \emph{The mathematics of computerized tomography}, B. G.
  Teubner, Stuttgart; John Wiley \& Sons, Ltd., Chichester, 1986. \MR{856916}

\bibitem[Nee10]{needell2010randomized}
Deanna Needell, \emph{Randomized {K}aczmarz solver for noisy linear systems},
  BIT \textbf{50} (2010), no.~2, 395--403.

\bibitem[NSV{\etalchar{+}}16]{GreedyKaczmarz}
Julie Nutini, Behrooz Sepehry, Alim Virani, Issam Laradji, Mark Schmidt, and
  Hoyt Koepke, \emph{Convergence rates for greedy {Kaczmarz} algorithms}, UAI
  (2016).

\bibitem[NSW16]{needell2016stochastic}
Deanna Needell, Nathan Srebro, and Rachel Ward, \emph{Stochastic gradient
  descent, weighted sampling, and the randomized {K}aczmarz algorithm}, Math.
  Program. \textbf{155} (2016), no.~1-2, 549--573.

\bibitem[NT14]{needell2014paved}
Deanna Needell and Joel~A Tropp, \emph{Paved with good intentions: analysis of
  a randomized block {K}aczmarz method}, Linear Algebra Appl. \textbf{441}
  (2014), 199--221.

\bibitem[Pop99]{popa1999block}
Constantin Popa, \emph{Block-projections algorithms with blocks containing
  mutually orthogonal rows and columns}, BIT \textbf{39} (1999), no.~2,
  323--338.

\bibitem[Pop01]{popa2001fast}
\bysame, \emph{A fast {K}aczmarz-{K}ovarik algorithm for consistent
  least-squares problems}, Korean J. Comp. App. Math. \textbf{8} (2001), no.~1,
  9--26.

\bibitem[PP15]{Petra2016}
Stefania Petra and Constantin Popa, \emph{Single projection {K}aczmarz extended
  algorithms}, Numer. Algorithms (2015), 1--16.

\bibitem[RM51]{robbins1951stochastic}
Herbert Robbins and Sutton Monro, \emph{A stochastic approximation method},
  Ann. Math. Stat. (1951), 400--407.

\bibitem[RN20]{rebrova2020block}
Elizaveta Rebrova and Deanna Needell, \emph{On block {G}aussian sketching for
  the {K}aczmarz method}, Numer. Algorithms (2020), 1--31.

\bibitem[RV15]{rudelson2015small}
Mark Rudelson and Roman Vershynin, \emph{Small ball probabilities for linear
  images of high-dimensional distributions}, Int. Math. Res. Notices
  \textbf{2015} (2015), no.~19, 9594--9617.

\bibitem[Sch73]{schlossmacher1973iterative}
EJ~Schlossmacher, \emph{An iterative technique for absolute deviations curve
  fitting}, J. Am. Stat. Assoc. \textbf{68} (1973), no.~344, 857--859.

\bibitem[SHS01]{savvides2001dynamic}
Andreas Savvides, Chih-Chieh Han, and Mani~B Strivastava, \emph{Dynamic
  fine-grained localization in ad-hoc networks of sensors}, Proc. Int. Conf. on
  Mobile computing and networking, 2001, pp.~166--179.

\bibitem[SS87]{SesSta}
M~Ibrahim Sezan and Henry Stark, \emph{Incorporation of a priori moment
  information into signal recovery and synthesis problems}, J. Math. Anal. Apl.
  \textbf{122} (1987), no.~1, 172--186.

\bibitem[SV09]{strohmer2009randomized}
Thomas Strohmer and Roman Vershynin, \emph{A randomized {K}aczmarz algorithm
  with exponential convergence}, J Fourier Anal. Appl. \textbf{15} (2009),
  no.~2, 262.

\bibitem[SZ13]{shamir2013stochastic}
Ohad Shamir and Tong Zhang, \emph{Stochastic gradient descent for non-smooth
  optimization: Convergence results and optimal averaging schemes}, Int. conf.
  on machine learning, 2013, pp.~71--79.

\bibitem[Ver18]{HDP}
Roman Vershynin, \emph{High-dimensional probability: An introduction with
  applications in data science}, vol.~47, Cambridge university press, 2018.

\bibitem[Wes81]{wesolowsky1981new}
GO~Wesolowsky, \emph{A new descent algorithm for the least absolute value
  regression problem: A new descent algorithm for the least absolute value},
  Commun. Stat. Simulat. \textbf{10} (1981), no.~5, 479--491.

\bibitem[WGZ06]{wang2006regularized}
Li~Wang, Michael~D Gordon, and Ji~Zhu, \emph{Regularized least absolute
  deviations regression and an efficient algorithm for parameter tuning}, Int.
  Conf. on Data Mining, IEEE, 2006, pp.~690--700.

\bibitem[WM10]{wright2010dense}
John Wright and Yi~Ma, \emph{Dense error correction via $\ell^1$-minimization},
  IEEE T. Infor. Theory \textbf{56} (2010), no.~7, 3540--3560.

\bibitem[ZF13]{Zouzias2013}
Anastasios Zouzias and Nikolaos~M. Freris, \emph{Randomized extended {Kaczmarz}
  for solving least squares}, SIAM J. Matrix Anal. A. \textbf{34} (2013),
  773--793.

\end{thebibliography}


\end{document}